\documentclass[reqno,b5paper]{amsart}
\usepackage{amsmath}
\usepackage{amssymb}
\usepackage{amsthm}
\usepackage{enumerate}
\usepackage[mathscr]{eucal}
\theoremstyle{plain}
\newtheorem{thm}{Theorem}[section]
\newtheorem{cor}[thm]{Corollary}
\newtheorem{lem}[thm]{Lemma}
\newtheorem{Example}{Example}[section]
\newtheorem{note}{Note}[section]

\theoremstyle{definition}
\newtheorem{defn}{Definition}[section]
\newtheorem{rem}{Remark}[section]

\begin{document}

\setcounter {page}{1}
\title{On strong $\mathcal{A}^\mathcal{I}$-statistical convergence of sequences in probabilistic metric spaces}
\author[ P. Malik and S. Das ]{Prasanta Malik* and Samiran Das*\ }
\newcommand{\acr}{\newline\indent}
\maketitle
\address{{*\,} Department of Mathematics, The University of Burdwan, Golapbag, Burdwan-713104,
West Bengal, India.
                Email: pmjupm@yahoo.co.in, das91samiran@gmail.com \acr
           }

\maketitle
\begin{abstract}
In this paper using a non-negative regular summability matrix $\mathcal{A}$ and a non-trivial admissible ideal $\mathcal{I}$ in $\mathbb{N}$ we study some basic properties of strong $\mathcal{A}^\mathcal{I}$-statistical convergence and strong $\mathcal{A}^\mathcal{I}$-statistical Cauchyness of sequences in probabilistic metric spaces not done earlier. We also introduce strong $\mathcal{A}^{\mathcal{I}^*}$-statistical Cauchyness in probabilistic metric space and study its relationship with strong $\mathcal{A}^{\mathcal{I}}$-statistical Cauchyness there. Further, we study some basic properties of strong $\mathcal{A}^\mathcal{I}$-statistical limit points and strong $\mathcal{A}^\mathcal{I}$-statistical cluster points of a sequence in probabilistic metric spaces. 
\end{abstract}

\author{}
\maketitle { Key words and phrases: probabilistic metric space, $\mathcal{A}^\mathcal{I}$-density, strong $\mathcal{A}^\mathcal{I}$-statistical convergence, strong $\mathcal{A}^\mathcal{I}$-statistical Cauchyness, strong $\mathcal{A}^{\mathcal{I}^*}$-statistical Cauchyness, strong $\mathcal{A}^\mathcal{I}$-statistical limit point, strong $\mathcal{A}^\mathcal{I}$-statistical cluster point.} \\


\section{\textbf{Introduction:}}\label{sec1}

The notion of ``statistical metric space" now known as Probabilistic metric space (in short PM space) was introduced by Menger \cite{Me} in 1942, as an important generalization of metric spaces. In PM space he used the distance between two points $u$ and $v$ as a distribution function $F_{uv}$ instead of a non-negative real number. 
After Menger, several mathematicians like Schwiezer and Sklar \cite{Sh1, Sh2, Sh3, Sh4}, Tardiff \cite{Tar}, Thorp \cite{Th} and many others, contributed a lot in the study of probabilistic metric spaces. A through discussion on the development of probabilistic metric spaces can be found in the well known book of Schwiezer and Sklar \cite{Sh5}. Many topologies are constructed on a PM space, but the strong topology is one, getting most of the importance to date and we are using it in this paper.

On the otherhand, the idea of statistical convergence was introduced as a generalization to the usual notion of convergence of real number sequences independently by Fast \cite{Fa} and Schoenberg \cite{Sc}, using the concept of natural density of subsets of $\mathbb{N}$, the set of all natural numbers. A set $\mathcal{K}\subset\mathbb{N}$ has natural density 
$d(\mathcal{K})$ if
$$d(\mathcal{K})=\lim\limits_{n\rightarrow \infty}\frac{\left|\mathcal{K}(n)\right|}{n}$$
where
$\mathcal{K}(n)=\left\{j\in \mathcal{K}:j\leq n\right\}$ and
 $\left|\mathcal{K}(n)\right|$ represents the number of elements in $\mathcal{K}(n)$.

A real number sequence $x=\{x_k\}_{k\in\mathbb{N}}$ is called statistically convergent to $\mathcal{L}\in\mathbb{R}$ if for every $\varepsilon>0,~d(B(\varepsilon))=0$ or $\lim\limits_{n\rightarrow\infty}(C_1\chi_{_{B(\varepsilon)}})_n=0$ where $B(\varepsilon)=\{k\in\mathbb{N}:\left|x_k-\mathcal{L}\right|\geq\varepsilon\}$ and $C_1$ is the Cesaro matrix of order $1$.

After the seminal works of \v{S}al\'{a}t \cite{Sa} and Fridy \cite{Fr1}, many more works on this convergence notion have been done which can be seen in \cite{Fr2, Pe, St}.

In 1981, the concept of
natural density was generalized to the notion of $\mathcal A$-density by Freedman et al. \cite{Fd}, using an arbitrary non-negative regular
summability matrix $\mathcal A$ in place of the Cesaro matrix $C_1$. An $\mathbb{N} \times \mathbb{N}$
matrix $\mathcal A = (a_{nk}),~ a_{nk} \in \mathbb{R}$ is called a regular summability matrix if for any convergent sequence $x=\{x_k\}_{k\in\mathbb N}$ of real numbers with limit $\xi$, $
\displaystyle{\lim_{n \rightarrow \infty}} \displaystyle{\sum_{ k
= 1}^{\infty}} a_{nk}x_k = \xi$, and $\mathcal A$ is called
non-negative if $ a_{nk} \geq 0, ~\forall n, k $. The well-known
Silverman-Toepliz's theorem asserts that an $\mathbb{N} \times
\mathbb{N}$ matrix $\mathcal A = (a_{nk}),~ a_{nk} \in \mathbb{R}$
is regular if and only if the following three conditions are
satisfied:

\begin{enumerate}[(i)]
\item
$\left\|\mathcal{A}\right\|=\sup\limits_{n}\sum\limits_{k}|a_{nk}|<\infty$,
\item $\lim\limits_{n\rightarrow\infty}a_{nk}=0$ for each $k$,
\item $\lim\limits_{n\rightarrow\infty}\sum\limits_{k}a_{nk}=1$.
\end{enumerate}

Throughout the paper we take $\mathcal A = (a_{nk})$ as an
$\mathbb{N} \times \mathbb{N}$ non-negative regular summability
matrix unless or otherwise mentioned.

A set $\mathcal{M}\subset\mathbb{N}$ is said to have $\mathcal{A}$ density $\delta_{\mathcal{A}}(\mathcal{M})$ if
 $$\delta_{\mathcal{A}}(\mathcal{M})=\lim\limits_{n\rightarrow \infty}\sum\limits_{m\in\mathcal{M}}a_{nm}.$$

The notion of statistical convergence was extended to the notion of $\mathcal{A}$-statistical convergence by Kolk \cite{Kl1}, using the notion of $\mathcal{A}$-density.

A sequence $x=\{x_k\}_{k\in\mathbb{N}}$ of real numbers is said to be $\mathcal{A}$-statistically convergent to $L\in \mathbb{R}$ if for every $\varepsilon>0$ we have
$$\delta_{\mathcal{A}}(\{k\in\mathbb{N}:\left|x_k-L\right|\geq\varepsilon\})=0,~\text{or},~
\lim\limits_{n\rightarrow\infty}\sum\limits_{\left|x_k-L\right|\geq\varepsilon}a_{nk}=0.$$
In this case we write $st_{\mathcal{A}}\mbox{-}\lim\limits_{k\rightarrow \infty}x_k = L$.
 The notion of $\mathcal{A}$-statistical limit point and $\mathcal{A}$-statistical cluster point were introduced by Connor et al. in \cite{Co4} and these notions were extensively studied in \cite{De1, De3}.

Further, the idea of statistical convergence was extended to $\mathcal{I}$-convergence by Kostyrko et al. \cite{Ko1} using the notion of an ideal $\mathcal{I}$ of subsets of $\mathbb{N}$. A non-empty class $\mathcal{I}\subset 2^X$, where $X\neq\emptyset$, is called an ideal, if the following three conditions are satisfied:
\begin{eqnarray*}
&(i)& \emptyset \in \mathcal{I};\\ 
&(ii)& A, B \in \mathcal{I} \Rightarrow A\cup B \in \mathcal{I};\\  
&(iii)& A\in \mathcal{I}, B\subset A \Rightarrow B\in \mathcal{I}.
\end{eqnarray*}
 An ideal $\mathcal{I}$ in $X$ is called non-trivial if $\mathcal{I}\neq\{\emptyset\}$ and $X\notin \mathcal{I}$. A non-trivial ideal $\mathcal{I}$ of $X$ is called admissible if for every $y\in X$, $\{y\}\in\mathcal{I}$. For a non-trivial ideal $\mathcal{I}$ in $X$ the filter associated with the ideal $\mathcal{I}$ is denoted by $\mathcal{F}(\mathcal{I})$ and is defined by $\mathcal{F}(\mathcal{I})=\{K\subset X:~\text{if there exists}~ H\in\mathcal{I}~\text{such that}~ X\setminus H=K\}$.

Throughout the paper $\mathcal{I}$ stands for a non-trivial admissible ideal in $\mathbb{N}$.
A real number sequence $x=\{x_k\}_{k \in \mathbb{N}}$ is said to be
$\mathcal{I}$-convergent to $\mathcal{L}\in\mathbb{R}$, if for every $\varepsilon >0$, $\{
k\in\mathbb{N} : \left| x_k - \mathcal{L} \right| \geq \varepsilon\} \in
\mathcal{I}$ and we write it as $\mathcal{I}\mbox{-}\lim\limits_{k\rightarrow \infty}x_k = \mathcal{L}$.

 Many more works on this line can be seen in \cite{Ko2, La1, La2}.

Recently the notion of $\mathcal{A}$-statistical convergence of real sequences was generalized to the notion of $\mathcal{A}^\mathcal{I}$-statistical convergence by Savas et al. \cite{Sav3} by using an ideal $\mathcal{I}$ in $\mathbb{N}$. The concept of $\mathcal{A}^\mathcal{I}$-statistical cluster point was introduced by G$\ddot{\text{u}}$rdal et al. \cite{Gu1}. 
 A set $\mathcal{M}\subset\mathbb{N}$ is said to have $\mathcal{A}^\mathcal{I}$-density $\delta_{\mathcal{A}^\mathcal{I}}(\mathcal{M})$ if
$$\delta_{\mathcal{A}^\mathcal{I}}(\mathcal{M})=\mathcal{I}\mbox{-}\lim\limits_{n\rightarrow \infty}(\mathcal{A}\chi_{\mathcal{M}})_n=\mathcal{I}\mbox{-}\lim\limits_{n\rightarrow\infty}\sum\limits_{m\in\mathcal{M}}a_{nm}.$$

From Lemma 2.4 of \cite{Pr4}, we have for $\mathcal{K}_1,\mathcal{K}_2\subset \mathbb{N}$ if $\delta_{\mathcal{A}^\mathcal{I}}(\mathcal{K}_1)$ and
$\delta_{\mathcal{A}^\mathcal{I}}(\mathcal{K})_2$ exist, then
\begin{enumerate}[(i)]
\item $\delta_{\mathcal{A}^\mathcal{I}}(\emptyset)=1$,
$\delta_{\mathcal{A}^\mathcal{I}}(\mathbb{N})=1$ and  $0\leq
\delta_{\mathcal{A}^\mathcal{I}}(\mathcal{K}_i)\leq 1$ for $i =
1,2$, 
\item $\left|\mathcal{K}_1\Delta\mathcal{K}_2\right|<\infty
\Rightarrow
\delta_{\mathcal{A}^\mathcal{I}}(\mathcal{K}_1)=\delta_{\mathcal{A}^\mathcal{I}}(\mathcal{K}_2)$,
\item $\mathcal{K}_1\cap\mathcal{K}_2=\emptyset \Rightarrow
\delta_{\mathcal{A}^\mathcal{I}}(\mathcal{K}_1)+\delta_{\mathcal{A}^\mathcal{I}}(\mathcal{K}_2)=
\delta_{\mathcal{A}^\mathcal{I}}(\mathcal{K}_1\cup\mathcal{K}_2)$,
\item
$\delta_{\mathcal{A}^\mathcal{I}}(\mathcal{K}_i^c)=1-\delta_{\mathcal{A}^\mathcal{I}}(\mathcal{K}_i)$
for $i = 1, 2$, 
\item
$\delta_{\mathcal{A}^\mathcal{I}}(\mathcal{K}_i)=0$ for $i=1,2
\Rightarrow
\delta_{\mathcal{A}^\mathcal{I}}\left(\bigcup\limits_{i=1}^2\mathcal{K}_i\right)=0$,
\item $\delta_{\mathcal{A}^\mathcal{I}}(\mathcal{K}_i)=1$ for
$i=1,2 \Rightarrow
\delta_{\mathcal{A}^\mathcal{I}}(\mathcal{K}_1\cap\mathcal{K}_2)=1,
\delta_{\mathcal{A}^\mathcal{I}}(\mathcal{K}_1\cup\mathcal{K}_2)=1$.
\end{enumerate}

If a real number sequence $x=\{x_k\}_{k\in \mathbb{N}}$ satisfies a property $\mathfrak{Q}$ for each $k$ except for a set of $\mathcal{A}^\mathcal{I}$-density zero, then we say $x$ satisfies the property $\mathfrak{Q}$ for ``almost all $k(\mathcal{A}^\mathcal{I})$'' and we write it in short as ``$a.a.k(\mathcal{A}^\mathcal{I})$''.

For an admissible ideal $\mathcal{I}$ of $\mathbb{N}$, the collection $\mathcal{J}({\mathcal{A}^\mathcal{I}}) = \{ B \subset \mathbb{N}:
\delta_{\mathcal{A}^\mathcal{I}} (B) = 0\}$, where $\mathcal A = (a_{nk})$ is
an $\mathbb{N} \times \mathbb{N}$ non-negative regular summability
matrix, forms an admissible ideal of $\mathbb{N}$ again.

A real number sequence $x=\{x_k\}_{k \in \mathbb{N}}$ is said to be $\mathcal{A}^\mathcal{I}$-statistically convergent to $\xi$
if for any $\varepsilon > 0$,
$$\delta_{\mathcal{A}^\mathcal{I}}(\{k\in\mathbb{N}:\left|x_k-\xi\right|\geq\varepsilon\})=0.$$
In this case we write $\mathcal{I}\mbox{-}st_{\mathcal{A}}\mbox{-}\lim\limits_{k\rightarrow \infty}x_k = \xi $ or simply as $x_k\xrightarrow{\mathcal{A}^\mathcal{I}\mbox{-}st}\xi$.
Many more works on this line can be seen in \cite{Ed1, Ed2, Pr4, Sav6}. 

A real number sequence $x=\{x_k\}_{k\in\mathbb{N}}$ is said to be $\mathcal{A}^\mathcal{I}$-statistically Cauchy if for every $\gamma>0$, there exists a natural number $m_0$ such that
$$\delta_{\mathcal{A}^\mathcal{I}}(\{k\in\mathbb{N}:\left|x_k-x_{m_0}\right|\geq\gamma\})=0.$$

Because of immense importance of probabilistic metric space in
applied mathematics, the notions of statistical convergence \cite
{Fa, Sc} and $\mathcal{I}$-convergence \cite{Ko1} were
extended to the setting of sequences in a PM space endowed with
the strong topology by \c{S}en\c{c}imen et al. in \cite{Se} and
\cite{Se2} respectively. Many more works on this line can be seen in \cite{Ag1, Ag2, Da5, Pr1}. The notion of strong $\mathcal{A}$-statistical convergence was studied by Malik et al. \cite{Pr3} in PM spaces. 
 It can be easily seen that the set of all $\mathcal{A}^\mathcal{I}$-density zero subsets of $\mathbb{N}$ forms an ideal $\mathcal{J}(\mathcal{A}^\mathcal{I})$ in $\mathbb{N}$. So the notions of strong $\mathcal{A}^\mathcal{I}$-statistical convergence and strong $\mathcal{A}^\mathcal{I}$-statistical Cauchyness are special cases of strong $\mathcal{I}$-convergence and strong $\mathcal{I}$-Cauchyness respectively in a PM space. Following Kostyrko et al. \cite{Ko1},
Bartoszewicz et al. \cite{Br} and \c{S}en\c{c}imen et al
\cite{Se2}, for an admissible ideal $\mathcal{I}$ of $\mathbb{N}$, if we consider the admissible ideal $\mathcal{J}({\mathcal{A}^\mathcal{I}})$, then the notions of strong $\mathcal{J}(\mathcal{A}^\mathcal{I})$-convergence, strong $\mathcal{J}(\mathcal{A}^\mathcal{I})$-Cauchyness,
strong $\mathcal{J}(\mathcal{A}^\mathcal{I})$-limit point and strong
$\mathcal{J}(\mathcal{A}^\mathcal{I})$-cluster point of sequences in a PM
space become the notions of strong $\mathcal{A}^\mathcal{I}$-statistical
convergence, strong $\mathcal{A}^\mathcal{I}$-statistical Cauchyness, strong
$\mathcal{A}^\mathcal{I}$-statistical limit point and strong $\mathcal{A}^\mathcal{I}$-statistical cluster point respectively.

In this paper we study some basic properties of strong $\mathcal{A}^\mathcal{I}$-statistical convergence, strong $\mathcal{A}^\mathcal{I}$-statistical Cauchyness, strong $\mathcal{A}^\mathcal{I}$-statistical limit points and strong $\mathcal{A}^\mathcal{I}$-statistical cluster points of a sequence in a probabilistic metric space not done earlier. We also introduce the notion of strong $\mathcal{A}^{\mathcal{I}^*}$-statistical Cauchyness and study its relationship with strong $\mathcal{A}^{\mathcal{I}}$-statistical Cauchyness.

\section{\textbf{Basic Definitions and Notations}}\label{sec2}
In this section, we first recall some basic concepts and results related to probabilistic metric (PM) spaces (for more details see in \cite{Sh1, Sh2, Sh3, Sh4, Sh5}).
\begin{defn}\cite{Sh5}\label{def21}
A monotonically non decreasing function $f: [-\infty, \infty] \rightarrow [0,1]$ is called a distribution function if $f(-\infty)=0$ and $ f(\infty)=1$.
\end{defn}
We denote the set of all distribution functions with left continuous over $(-\infty,\infty)$ by $\mathcal{D}$. The relation $\leq$ on $\mathcal{D}$ defined by $f \leq g$ if and only if $f(a)\leq g(a),~\forall~ a \in [-\infty, \infty]$ is clearly a partial order relation on $\mathcal{D}$.

If $b\in [-\infty,+\infty]$, then the unit step at $b$ is defined on $\mathcal{D}$ by 
\[ \varepsilon_{b}(a) = \left\{
  \begin{array}{l l}
    0, & \quad a\in [-\infty,b)\\
    1, & \quad a\in (b,+\infty].
  \end{array} \right.\]

\begin{defn}\cite{Sh5}\label{def22}
A sequence $\{f_k\}_{k\in \mathbb{N}}$ in $\mathcal{D}$ is said to converge weakly to $f\in\mathcal{D}$ written as $f_k\xrightarrow{w}f$, if $\{f_k(\xi)\}_{k\in\mathbb{N}}$ converges to $f(\xi)$ at each continuity point $\xi$ of $f$.
\end{defn}

If $f,g\in \mathcal{D}$, then the distance $d_{L}(f,g)$ between $f$ and $g$ is defined by the infimum of all numbers $a\in(0,1]$ such that
\begin{eqnarray*}
&&~~~f(\xi-a)-a\leq g(\xi)\leq f(\xi+a)+a \\
&\text{and} &~g(\xi-a)-a\leq f(\xi)\leq g(\xi+a)+a,\hspace{0.1in}\text{holds}~\forall\xi\in\left(-\frac{1}{a},\frac{1}{a}\right).
\end{eqnarray*}

Then $(\mathcal{D},d_L)$ forms a metric space with the metric $d_{L}$. Clearly if $\{f_k\}_{k\in\mathbb{N}}$ is a sequence in $\mathcal{D}$ and
$f\in\mathcal{D}$, then $f_k\xrightarrow{w}f ~\text{if and only if}~ d_{L}(f_k,f)\rightarrow 0$.

\begin{defn}\cite{Sh5}\label{def23}
A non decreasing function $f:[0,\infty]\longrightarrow\mathbb{R}$, left continuous on $(0,\infty)$ is said to be a distance distribution function if $f(0)=0$ and $ f(\infty)=1$.
\end{defn}
We denote the set consisting of all the distance distribution functions as $\mathcal{D}^{+}$. Clearly $(\mathcal{D}^{+},d_L)$ is a compact metric space and thus complete.

\begin{thm}\cite{Sh5}\label{thm21}
If $f\in\mathcal{D}^{+}$ then for any $t>0$, $f(t)>1-t$ if and only if $d_{L}(f,\varepsilon_{0})<t$.
\end{thm}

\begin{defn}\cite{Sh5}\label{def24}
A triangle function is a binary operation $\tau$ on $\mathcal{D}^{+}$, which is associative, commutative, nondecreasing in each place and $\varepsilon_{0}$ is the identity element.
\end{defn}

\begin{defn}\cite{Sh5}\label{def25}
A probabilistic metric space, in short PM space, is a triplet $(X,\mathcal{F},\tau)$ where $X$ is a nonempty set whose elements are the points of the space; $\mathcal{F}$ is a function from $X\times X$ into $\mathcal{D}^{+}$, $\tau$ is a triangle function and the following conditions are satisfied for all $a,b,c\in X$:
    \begin{eqnarray*}
    &\textbf{(P\mbox{-}1).}&\mathcal{F}(a,a)=\varepsilon_{0},\\
    &\textbf{(P\mbox{-}2).}&\mathcal{F}(a,b)\neq\varepsilon_{0} ~\text{if}~ a\neq b,\\
    &\textbf{(P\mbox{-}3).}&\mathcal{F}(a,b)=\mathcal{F}(b,a),\\
    &\textbf{(P\mbox{-}4).}&\mathcal{F}(a,c)\geq \tau(\mathcal{F}(a,b),\mathcal{F}(b,c)).
    \end{eqnarray*}
  
Henceforth we will denote $\mathcal{F}(a,b)$ by $\mathcal{F}_{ab}$ and its value at $t$ by $\mathcal{F}_{ab}(t)$.
\end{defn}

\begin{Example}\cite{Sh5}\label{exa21}
Let $F\in\mathcal{D}^+$ is different from $\varepsilon_0$ and $\varepsilon_\infty$. Then $(X,\mathcal{F},M)$ is a equilateral PM space where $\mathcal{F}_{uv}$ is given by
\[\mathcal{F}_{uv}=\left\{
\begin{array}{l l}
F, & \quad\text{if}~ u\neq v\\
\varepsilon_0, & \quad \text{if}~ u=v
\end{array} \right.\]
and $M$ is the maximal triangle function.
\end{Example}

\begin{defn}\cite{Sh5}\label{def26}
Let $(X,\mathcal{F},\tau)$ be a PM space. For $\xi\in X$ and $t>0$, the strong $t$-neighborhood of $\xi$ is denoted by $\mathcal{N}_{\xi}(t)$ and is defined by
\begin{center}
$\mathcal{N}_{\xi}(t)=\{\eta\in X : \mathcal{F}_{\xi\eta}(t)>1-t\} $.
\end{center}
The collection $\mathfrak{N}_{\xi}=\{\mathcal{N}_{\xi}(t):t>0 \}$ is called the strong neighborhood system at $\xi$ and the union $\mathfrak{N}=\bigcup\limits_{\xi\in X}\mathfrak{N}_{\xi}$ is called the strong neighborhood system for $X$.
\end{defn}

From Theorem \ref{thm21}, we can write $\mathcal{N}_{a}(r)=\{b\in X:
d_{L}(\mathcal{F}_{ab},\varepsilon_{0})<r\} $. If $\tau$ is continuous, then the strong neighborhood system $\mathfrak{N}$ determines a Hausdorff topology for $X$. This topology is called the strong topology for $X$ and members of this topology are called strong open sets. Clearly, $\mathcal{N}_{\beta}(t)$ where $\beta\in X$, $t>0$ is a basic open set of this strong topology.

Throughout the paper, in a PM space $(X,\mathcal{F},\tau)$, we always consider that $\tau$ is continuous and $X$ is endowed with the strong topology.

In a PM space $(X,\mathcal{F},\tau)$ the strong closure of any subset $\mathcal{M}$ of $X$ is denoted by $k(\mathcal{M})$ and for any subset $\mathcal{M}(\neq\emptyset)$ of $X$ strong closure of $\mathcal{M}$ is defined by,
$$k(\mathcal{M})=\{c\in X: ~\text{for any}~ t>0, ~\exists~ e\in \mathcal{M} ~\text{such that}~\mathcal{F}_{ce}(t)>1-t\}.$$ 

\begin{defn}\cite{Du1}\label{def27}
Let $(X,\mathcal{F},\tau)$ be a PM space. Then a subset $\mathcal{M}$ of $X$ is called strongly closed if its complement is a strongly open set.
\end{defn}

\begin{defn}\cite{Pr1, Se}\label{def28}
Let $(X,\mathcal{F},\tau)$ be a PM space and $\mathcal{M}\neq \emptyset$ be a subset of $X$. Then $l\in X$ is said to be a strong limit point of $\mathcal{M}$ if for every $t>0$,
$$\mathcal{N}_l(t)\cap(\mathcal{M}\setminus\{l\})\neq\emptyset.$$
The set of all strong limit points of the set $\mathcal{M}$ is denoted by $L_\mathcal{M}^\mathcal{F}$.
\end{defn}

\begin{defn}\cite{Du1}\label{def29}
Let $(X,\mathcal{F},\tau)$ be a PM space and $\mathcal{M}$ be a subset of $X$. Let $\mathfrak{Q}$ be a family of strongly open subsets of $X$ such that $\mathfrak{Q}$ covers $\mathcal{M}$. Then $\mathfrak{Q}$ is said to be a strong open cover for $\mathcal{M}$.
\end{defn}

\begin{defn}\cite{Du1}\label{def210}
Let $(X,\mathcal{F},\tau)$ be a PM space and $\mathcal{M}$ be a subset of $X$. Then $\mathcal{M}$ is called strongly compact set if for every strong open cover of $\mathcal{M}$ has a finite subcover.
\end{defn}

\begin{defn}\cite{Du1}\label{def211}
Let $(X,\mathcal{F},\tau)$ be a PM space and $x=\{x_k\}_{k\in\mathbb{N}}$ be a sequence in $X$. Then $x$ is said to be strongly bounded if there exists a strongly compact subset $\mathcal{C}$ of $X$ such that $x_k\in \mathcal{C}$, $\forall ~k\in\mathbb{N}$. 
\end{defn}

\begin{defn}\cite{Du1}\label{def212}
Let $(X,\mathcal{F},\tau)$ be a PM space and $x=\{x_k\}_{k\in\mathbb{N}}$ be a sequence in $X$. Then $x$ is said to be strongly statistically bounded if there exists a strongly compact subset $\mathcal{C}$ of $X$ such that $d(\{k\in\mathbb{N}:x_k\notin \mathcal{C}\})=0$. 
\end{defn}

\begin{thm}\cite{Du1}\label{thm22} 
Let $(X,\mathcal{F},\tau)$ be a PM space and $\mathcal{M}$ be a strongly compact subset of $X$. Then every strongly closed subset of $\mathcal{M}$ is strongly compact.
\end{thm}

\begin{defn}\cite{Sh5}\label{def213}
Let $(X,\mathcal{F},\tau)$ be a PM space. Then for any $u>0$, the subset $\mathcal{V}(u)$ of $ X\times X$ given by
\begin{center}
$\mathcal{V}(u)=\{(p,q):\mathcal{F}_{pq}(u)>1-u\} $
\end{center}
is called the strong $u$-vicinity.
\end{defn}

\begin{thm}\cite{Sh5}\label{thm23}
 Let $(X,\mathcal{F},\tau)$ be a PM space and $\tau $ be continuous. Then for any $u>0$, there is an $\alpha>0$ such that
 $\mathcal{V}(\alpha)\circ\mathcal{V}(\alpha)\subset
 \mathcal{V}(u)$, where $\mathcal{V}(\alpha)\circ\mathcal{V}(\alpha)=\{(p,r):$ ~\mbox{for some}~ $q$,~  $(p,q)$
 and $(q,r)\in \mathcal{V}(\alpha)\}$.
 \end{thm}

From the hypothesis of Theorem \ref{thm23}, we can say that for any $u>0$,
there is an $\alpha >0 $ such that $\mathcal{F}_{pr}(u)>1-u$
 whenever $\mathcal{F}_{pq}(\alpha)>1-\alpha $ and $\mathcal{F}_{qr}(\alpha)>1-\alpha$. Equivalently it
 can be written as: for any $u>0$, there is an $\alpha>0$ such that
 $d_{L}(\mathcal{F}_{pr},\varepsilon_{0})<u$ whenever $d_{L}(\mathcal{F}_{pq},\varepsilon_{0})<\alpha$
 and $ d_{L}(\mathcal{F}_{qr}, \varepsilon_{0})<\alpha$.

\begin{defn}\cite{Se}\label{def214}
Let $(X,\mathcal{F},\tau)$ be a PM space. A sequence $x=\{x_k\}_{k\in \mathbb{N}}$ in $X$ is said to be strongly convergent to
$\mathcal{L}\in X$ if for every $t>0$, $\exists$ a natural number $k_0$ such that
$$x_k\in\mathcal{N}_\mathcal{L}(t),\hspace{1 in} \text{whenever}~ k\geq k_0 .$$
\end{defn}
In this case, we write
$\mathcal{F}$-$\lim\limits_{k\rightarrow\infty}x_k=\mathcal{L}$ or, $x_k\stackrel{\mathcal{F}}\longrightarrow \mathcal{L}$.

\begin{defn}\cite{Sh4}\label{def215}
Let $(X,\mathcal{F},\tau)$ be a PM space. A sequence $x=\{x_k\}_{k\in\mathbb{N}}$ in $X$ is said to be strongly Cauchy if for every $t>0$, there exists a natural number $k_0$ such that
$$(x_k,x_r)\in\mathcal{U}(t), \hspace{1 in} \text{whenever}~ k,r\geq k_0.$$
\end{defn}

\begin{rem}
The Definition \ref{def215} can be restated as follows: A sequence $x=\{x_k\}_{k\in\mathbb{N}}$ in a PM space $(X,\mathcal{F},\tau)$ is said to be strongly Cauchy if for every $t>0$, there exists a natural number $k_0=k_0(t)$ such that
$$(x_k,x_{k_0})\in\mathcal{U}(t), \hspace{1 in} \text{whenever}~ k\geq k_0.$$
\end{rem}

\begin{defn}\cite{Se}\label{def216}
Let $(X,\mathcal{F},\tau)$ be a PM space. A sequence $x=\{x_k\}_{k \in \mathbb{N}}$ in $X$ is said to be strongly statistically convergent to $\alpha\in X$ if for any $t> 0$
\begin{center}
$d(\{k\in\mathbb{N}: \mathcal{F}_{x_k\alpha}(t)\leq 1-t\})=0,
~~~~~~~~\text{or},~~~~~~~~ d(\{ k\in\mathbb{N}:x_k\notin\mathcal{N}_\alpha(t)\})=0.$
\end{center}
In this case we write $st^{\mathcal{F}}$-$\lim\limits_{k\rightarrow\infty}x_k = \alpha$.
\end{defn}

\begin{defn}\cite{Se}\label{def217}
Let $(X,\mathcal{F},\tau)$ be a PM space. A sequence
$x=\{x_k\}_{k \in \mathbb{N}}$ in $X$ is said to be strongly statistically Cauchy if for any $t>0$, $\exists$ a natural number $N_0=N_0(t)$ such that 
$$d(\{k\in\mathbb{N}:\mathcal{F}_{x_kx_{N_0}}(t)\leq 1-t\})=0,~\text{or},~ d(\{k\in\mathbb{N}:x_k\notin\mathcal{N}_{x_{N_0}}(t)\})=0.$$
\end{defn}

\section{\textbf{Strong $\mathcal{A}^\mathcal{I}$-statistical convergence and strong $\mathcal{A}^\mathcal{I}$-statistical Cauchyness}}\label{sec3}

In this section, following the works of Das et al. \cite{Da5}, \c{S}en\c{c}imen et al. \cite{Se2} and Malik et al. \cite{Pr3} we study the notions of strong $\mathcal{A}^\mathcal{I}$-statistical convergence and strong $\mathcal{A}^\mathcal{I}$-statistical Cauchyness in a PM space. 

\begin{defn}\cite{Se2}\label{def220}
Let $(X,\mathcal{F},\tau)$ be a PM space and $x=\{x_k\}_{k\in \mathbb{N}}$ be a sequence in $X$. Then $x$ is said to be strongly $\mathcal{I}$-convergent to $\mathcal{L} \in X$, if for every $t>0$, the set
$$\{k\in \mathbb{N}: \mathcal{F}_{x_k\mathcal{L}}(t)\leq 1-t\}\in \mathcal{I},~\text{or},~\{k\in \mathbb{N}: x_k\notin\mathcal{N}_\mathcal{L}(t)\}\in \mathcal{I}.$$
\end{defn}
In this case we write $\mathcal
I^{\mathcal{F}}$-$\lim\limits_{k\rightarrow\infty}x_k = \mathcal{L}$.

\begin{defn}\cite{Se2}\label{def221}
Let $(X,\mathcal{F},\tau)$ be a PM space and $x=\{x_k\}_{k\in \mathbb{N}}$ be a sequence in $X$. Then $x$ is said to be strongly $\mathcal{I}$-Cauchy sequence if for every $t>0$, $\exists$ a natural number $k_0$ depending on $t$ such that, the set
$$\{k\in \mathbb{N}: \mathcal{F}_{x_kx_{k_0}}(t)\leq 1-t\}\in \mathcal{I},~\text{or},~\{k\in \mathbb{N}: x_k\notin\mathcal{N}_{x_{k_0}}(t)\}\in \mathcal{I}.$$
\end{defn}

\begin{note}\label{no21}
(i) If $\mathcal{I}=\mathcal{I}_{fin}=\{\mathcal{K}\subset
\mathbb{N}: \left|\mathcal{K}\right|<\infty\}$, then in a PM space
the notions of strong $\mathcal {I}_{fin}$-convergence and strong
$\mathcal {I}_{fin}$-Cauchyness coincide with the notions of
strong convergence and strong Cauchyness respectively.

(ii) If $\mathcal{I}=\mathcal{I}_{d}=\{\mathcal{K}\subset
\mathbb{N}: d(\mathcal{K})=0\}$, then in a PM space
the notions of strong $\mathcal {I}_{d}$-convergence and strong
$\mathcal {I}_{d}$-Cauchyness coincide with the notions of
strong statistical convergence \cite{Se} and strong statistical Cauchyness \cite{Se} respectively.

(iii) Let $\mathcal{I}$ be an admissible ideal in $\mathbb{N}$ then the notions of strong
$\mathcal{J}(\mathcal{A}^\mathcal{I})$-convergence and strong
$\mathcal{J}(\mathcal{A}^\mathcal{I})$-Cauchyness of sequences in a PM space
coincide with the notions of strong $\mathcal{A}^\mathcal{I}$-statistical
convergence and strong $\mathcal{A}^\mathcal{I}$-statistical Cauchyness
respectively. Further, if $\mathcal{I}=\mathcal{I}_{fin}=\{\mathcal{K}\subset\mathbb{N}: \left|\mathcal{K}\right|<\infty\}$, then the notions of strong $\mathcal{J}(\mathcal{A}^{\mathcal{I}_{fin}})$-convergence and strong $\mathcal{J}(\mathcal{A}^{\mathcal{I}_{fin}})$-Cauchyness of sequences in a PM space coincide with strong $\mathcal{A}$-statistical convergence \cite{Pr3} and strong $\mathcal{A}$-statistical Cauchyness \cite{Pr3} respectively.

(iv) If $\mathcal{A}$ is the Cesaro matrix $C_1$ and $\mathcal{I}$ is an admissible ideal,
then the notions of strong $\mathcal{J}({C_1}^\mathcal{I})$-convergence and
strong $\mathcal{J}({C_1}^\mathcal{I})$-Cauchyness of sequences in a PM space
coincide with the notions of strong $\mathcal{I}$-statistical convergence \cite{Da5} and
strong $\mathcal{I}$-statistical Cauchyness \cite{Da5} respectively.
\end{note}

In view of Definition \ref{def220}, Definition \ref{def221} and Note \ref{no21}.(iii) we now restate the definitions of strong $\mathcal{A}^\mathcal{I}$-statistical convergence and strong $\mathcal{A}^\mathcal{I}$-statistical Cauchyness in a PM space.

\begin{defn}\label{def31}\cite{Se2}
Let $(X,\mathcal{F},\tau)$ be a PM space and $x=\{x_k\}_{k\in\mathbb{N}}$ be a sequence in $X$. Then $x$ is said to be strongly $\mathcal{A}^\mathcal{I}$-statistically convergent to $\mathcal{L} \in X$, if for every $t>0$,
$$\delta_{\mathcal{A}^\mathcal{I}}(\{k\in \mathbb{N}: \mathcal{F}_{x_k\mathcal{L}}(t)\leq 1-t\})=0~\text{or,}~
\delta_{\mathcal{A}^\mathcal{I}}(\{k\in \mathbb{N}: x_k \notin \mathcal{N}_\mathcal{L}(t)\})=0.$$
We write it as,
$\mathcal{I}\mbox{-}st_\mathcal{A}^{\mathcal{F}}$-$\lim\limits_{k\rightarrow \infty}x_k = \mathcal{L} $ or simply as
$x_k\xrightarrow{\mathcal{A}^\mathcal{I}\mbox{-}st^{\mathcal{F}}}\mathcal{L}$. $\mathcal{L}$ is called the strong $\mathcal{A}^\mathcal{I}$-statistical limit of $x$.
\end{defn}

\begin{defn}\label{def32a}\cite{Se2}
Let $(X,\mathcal{F},\tau)$ be a PM space and $x=\{x_k\}_{k\in\mathbb{N}}$ be a sequence in $X$. Then $x$ is said to be strongly $\mathcal{A}^\mathcal{I}$-statistically Cauchy sequence if for every $t>0$, there exists a natural number $k_0$ depending on $t$ such that 
$$\delta_{\mathcal{A}^\mathcal{I}}(\{k\in \mathbb{N}: \mathcal{F}_{x_kx_{k_0}}(t)\leq 1-t\})=0,~
\text{or},~\delta_{\mathcal{A}^\mathcal{I}}(\{k\in \mathbb{N}: x_k\notin\mathcal{N}_{x_{k_0}}(t)\})=0.$$
\end{defn}

\begin{rem}\label{rem31}
The following three statements are equivalent:
\begin{enumerate}[(i)]
    \item $x_k\xrightarrow{\mathcal{A}^\mathcal{I}\mbox{-}st^{\mathcal{F}}}\mathcal{L}$
    \item For each $t>0,~\delta_{\mathcal{A}^\mathcal{I}}(\{k\in \mathbb{N}: d_L(\mathcal{F}_{x_k\mathcal{L}},\varepsilon_0)\geq t\})=0$
    \item $\mathcal{I}\mbox{-}st_\mathcal{A}^{\mathcal{F}}\mbox{-}\lim\limits_{k\rightarrow\infty}d_L(\mathcal{F}_{x_k\mathcal{L}},\varepsilon_0)=0$.
\end{enumerate}
\end{rem}
\begin{proof}
It is clear from Theorem \ref{thm21} and the Definition \ref{def31}. 
\end{proof}

\begin{thm}\label{thm31}
Let $(X,\mathcal{F},\tau)$ be a PM space and $x=\{x_k\}_{k\in\mathbb{N}}$ be a strongly $\mathcal{A}^\mathcal{I}$-statistically convergent sequence in $X$. Then strong $\mathcal{A}^\mathcal{I}$-statistical limit of $x$ is unique.
\end{thm}
\begin{proof}
If possible, let $\mathcal{I}\mbox{-}st_\mathcal{A}^{\mathcal{F}}$-$\lim\limits_{k\rightarrow\infty}x_k = \alpha_1 $ and $\mathcal{I}\mbox{-}st_\mathcal{A}^{\mathcal{F}}$-$\lim\limits_{k\rightarrow \infty}x_k = \alpha_2$ with $\alpha_1\neq \alpha_2$. So $\mathcal{F}_{\alpha_1\alpha_2}\neq\varepsilon_0$. Then there is a $t>0$ such that $d_L(\mathcal{F}_{\alpha_1\alpha_2},\varepsilon_0)=t$. We choose $\gamma>0$ so that $d_L(\mathcal{F}_{pq},\varepsilon_0)<\gamma$ and $d_L(\mathcal{F}_{qr},\varepsilon_0)<\gamma $ imply that $d_L(\mathcal{F}_{pr},\varepsilon_0)<t$. Since $\mathcal{I}\mbox{-}st_\mathcal{A}^{\mathcal{F}}$-$\lim\limits_{k\rightarrow \infty}x_k = \alpha_1 $ and $\mathcal{I}\mbox{-}st_\mathcal{A}^{\mathcal{F}}$-$\lim\limits_{k\rightarrow\infty}x_k = \alpha_2$, so $ \delta_{\mathcal{A}^\mathcal{I}}(Z_1(\gamma)) = 0$ and $\delta_{\mathcal{A}^\mathcal{I}}(Z_2(\gamma)) = 0$, where
$$Z_1(\gamma)=\{k\in\mathbb{N}: x_k\notin\mathcal{N}_{\alpha_1}(\gamma)\}$$ and
$$Z_2(\gamma)=\{k\in \mathbb{N}: x_k\notin\mathcal{N}_{\alpha_2}(\gamma)\}.$$ 
Now, let $Z_3(\gamma)=Z_1(\gamma)\cup Z_2(\gamma)$. Then $\delta_{\mathcal{A}^\mathcal{I}}(Z_3(\gamma))=0$ and so
$\delta_{\mathcal{A}^\mathcal{I}}(Z^c_3(\gamma))=1$. Let $k\in Z^c_3(\gamma).$ Then
$d_L(\mathcal{F}_{x_k\alpha_1},\varepsilon_0)<\gamma$ and
$d_L(\mathcal{F}_{\alpha_2x_k},\varepsilon_0)<\gamma$ and so
$d_L(\mathcal{F}_{\alpha_1\alpha_2},\varepsilon_0)<t$, a contradiction. Hence strong $\mathcal{A}^\mathcal{I}$-statistical limit of a strongly $\mathcal{A}^\mathcal{I}$-statistically convergent sequence in a PM space is unique.
\end{proof}

\begin{thm}\label{thm32}
Let $(X,\mathcal{F},\tau)$ be a PM space and
$x=\{x_k\}_{k\in\mathbb{N}}$, $y=\{y_k\}_{k\in\mathbb{N}}$ be two
sequences in $X$ such that
$x_k\xrightarrow{\mathcal{A}^\mathcal{I}\mbox{-}st^{\mathcal{F}}} p \in X $ and
$y_k\xrightarrow{\mathcal{A}^\mathcal{I}\mbox{-}st^{\mathcal{F}}} q \in X$. Then
$$\mathcal{I}\mbox{-}st_\mathcal{A}^{\mathcal{F}}\mbox{-}\lim\limits_{k\rightarrow\infty}d_L(\mathcal{F}_{x_ky_k}, \mathcal{F}_{pq})=0.$$
\end{thm}
\begin{proof}
The proof directly follows from [Theorem 3.1 \cite{Se2}], by taking the ideal
$ \mathcal{J}(\mathcal{A}^\mathcal{I})$.
\end{proof}

\begin{thm}\label{thm33}
Let $(X,\mathcal{F},\tau)$ be a  PM space and $x=\{x_k\}_{k\in\mathbb{N}}$ be a sequence in $X$. If $x$ is strongly convergent to $\mathcal{L} \in X$, then $\mathcal{I}\mbox{-}st_\mathcal{A}^{\mathcal{F}}\mbox{-}\lim\limits_{k\rightarrow\infty} x_k=\mathcal{L}$.
\end{thm}
\begin{proof}
The proof is trivial and so is omitted.
\end{proof}

We now using the condtion AP$\mathcal{A}^\mathcal{I}$O in \cite{Pr4} to prove some useful results discussed in \cite{Pr1}.
\begin{defn}\textbf{(Additive property for $\mathcal{A}^\mathcal{I}$-density zero sets)} \cite{Pr4}\label{def32}
The $\mathcal{A}^\mathcal{I}$-density $\delta_{\mathcal{A}^\mathcal{I}}$ is said to satisfy the condition AP$\mathcal{A}^\mathcal{I}$O if, given any countable collection of
mutually disjoint sets $\{G_j\}_{j\in\mathbb N}$ in $\mathbb N$ with $\delta_{\mathcal{A}^\mathcal{I}}(G_j)=0$ for
each $j\in\mathbb N$, there exists a collection of sets $\{H_j\}_{j\in\mathbb N}$ in $\mathbb N$ with
the properties $\left|G_j\Delta H_j\right|<\infty$ for each $j\in\mathbb N$ and $\delta_{\mathcal{A}^\mathcal{I}}(H=\bigcup\limits_{j=1}^\infty H_j)=0$.
\end{defn}

\begin{defn}\label{def33}\cite{Se2}
Let $(X,\mathcal{F},\tau)$ be a PM space and $x=\{x_k\}_{k\in\mathbb{N}}$ be a sequence in $X$. Then $x$ is said to be strongly $\mathcal{I}^*$-convergent to $\mathcal{L}$ in $X$, if there exists a set $\mathcal{K}=\{k_1<k_2<...<k_j<...\}(\subset\mathbb{N})\in\mathbb{F}(\mathcal{I})$ such that $\mathbb{N}\setminus\mathcal{K}\in\mathcal{I}$ and the subsequence $\{x\}_{\mathcal{K}}$ strongly convergent to $\mathcal{L}$.
\end{defn}

\begin{note}\label{no31}
If we take $\mathcal I = \mathcal{J}(\mathcal{A}^\mathcal{I})$ then the notion of strong $\mathcal{I}^*$-convergence becomes the notion of strong $\mathcal{A}^{\mathcal{I}^*}$-statistical convergence.
\end{note}

In view of Definition \ref{def33} and Note \ref{no31} we restate the definition of strong $\mathcal{A}^{\mathcal{I}^*}$-statistical convergence.

\begin{defn}\cite{Se2}\label{def34}
Let $(X,\mathcal{F},\tau)$ be a PM space and $x=\{x_k\}_{k\in\mathbb{N}}$ be a sequence in $X$. Then $x$ is said to be strongly $\mathcal{A}^{\mathcal{I}^*}$-statistically convergent to $\mathcal{L}$ in $X$, if there exists a set $\mathcal{K}=\{k_1<k_2<...<k_j<...\}(\subset\mathbb{N})\in\mathbb{F}(\mathcal{I})$ such that $\delta_{\mathcal{A}^\mathcal{I}}(\mathcal{K})=1$ and the subsequence $\{x\}_{\mathcal{K}}$ strongly convergent to $\mathcal{L}$. In this case we write $x_k\xrightarrow{\mathcal{A}^{\mathcal{I}^*}\mbox{-}st^\mathcal{F}}\mathcal{L}$ and $\mathcal{L}$ is called strong $\mathcal{A}^{\mathcal{I}^*}$-statistical limit of $x$. 
\end{defn}

\begin{defn}\label{def35}
Let $(X,\mathcal{F},\tau)$ be a PM space and $x=\{x_k\}_{k\in\mathbb{N}}$ be a sequence in $X$. Then $x$ is said to be strong $\mathcal{A}^{\mathcal{I}^*}$-statistically Cauchy sequence if there exists a set $\mathcal{K}=\{k_1<k_2<...<k_j<...\}(\subset\mathbb{N})\in\mathbb{F}(\mathcal{I})$ such that $\delta_{\mathcal{A}^\mathcal{I}}(\mathcal{K})=1$ and the subsequence $\{x\}_{\mathcal{K}}$ strongly Cauchy in $X$.
\end{defn}

\begin{thm}\label{thm34}
Let $(X,\mathcal{F},\tau)$ be a  PM space, $x=\{x_k\}_{k\in\mathbb{N}}$ be a sequence in $X$ and $\mathcal{I}$ be an admissible ideal in $\mathbb N$ such that $\delta_{\mathcal{A}^\mathcal{I}}$ has the property AP$\mathcal{A}^\mathcal{I}$O. Then $x$ is strongly $\mathcal{A}^\mathcal{I}$-statistically convergent to $\mathcal{L}$ if and only if $x$ is strongly $\mathcal{A}^{\mathcal{I}^*}$-statistically convergent to $\mathcal{L}$.
\end{thm}
\begin{proof}
Let $x=\{x_k\}_{k\in\mathbb{N}}$ be a sequence in $X$ such that $x$ is strongly $\mathcal{A}^\mathcal{I}$-statistically convergent to $\mathcal{L}\in X$. Then for all $t>0$, the set $\{k\in\mathbb N:\mathcal{F}_{x_k\mathcal{L}}(t)\leq 1-t\}$
has $\mathcal{A}^\mathcal{I}$-density zero. Then, $$\delta_{\mathcal{A}^\mathcal{I}}(\{k\in\mathbb{N}:d_L(\mathcal{F}_{x_k\mathcal{L}},\varepsilon_0)\geq t\})=0.$$ 
Set
$G_1=\{k\in\mathbb{N}:d_L(\mathcal{F}_{x_k\mathcal{L}},\varepsilon_0)\geq 1\}$, $G_j=\{k\in\mathbb{N}:\frac{1}{j-1}> d_L(\mathcal{F}_{x_k\mathcal{L}},\varepsilon_0)\geq\frac{1}{j}\}$ for $ j \geq 2, j \in
\mathbb{N}$. Then $\{G_j\}_{j\in\mathbb N}$ is a sequence of mutually disjoint subsets of $\mathbb{N}$ with $\delta_{\mathcal{A}^\mathcal{I}}(G_j)=0$ for each $j\in\mathbb{N}$. Since $\delta_{\mathcal{A}^\mathcal{I}}$ satisfies the property $AP\mathcal{A}^\mathcal{I}O$ so there exists a sequence $\{H_j\}_{j\in\mathbb N}$ of subsets of $\mathbb{N}$ with $\left|G_j\Delta
H_j\right|<\infty$ and $\delta_{\mathcal{A}^\mathcal{I}}(H=\bigcup\limits_{j=1}^\infty H_j)=0$.
We claim that $\lim\limits_{\stackrel{\stackrel{k\in \mathcal{M}}{k\rightarrow \infty}}~}x_k=\mathcal{L}$ where $\mathcal{M}=\mathbb{N}\setminus H$. To prove our claim, let $\gamma>0$ be given. Choose $i\in \mathbb{N}$ so that $\frac{1}{i+1}<\gamma $. Then
$\{k\in\mathbb N: d_L(\mathcal{F}_{x_k\mathcal{L}},\varepsilon_0)\geq \gamma\}\subset
\bigcup\limits_{j=1}^{i+1}G_j$. Since $\left|G_j\Delta
H_j\right|<\infty$ for all $j=1,2,...,i+1$, so there exists $n'\in\mathbb N$
such that
$\bigcup\limits_{j=1}^{i+1}G_j\cap(n',\infty)=\bigcup\limits_{j=1}^{i+1}H_j\cap(n',\infty)$.
Now if $k\notin H$, $k>n'$ then $k\notin
\bigcup\limits_{j=1}^{i+1}H_j$ and consequently, $k\notin
\bigcup\limits_{j=1}^{i+1}G_j$, which implies
$d_L(\mathcal{F}_{x_k\mathcal{L}},\varepsilon_0)<\gamma$. Therefore, $x$ is strongly $\mathcal{A}^{\mathcal{I}^*}$-statistically convergent to $\mathcal{L}$.

The proof of the converse part directly follows from [Theorem 3.2 \cite{Se2}], by taking the ideal
$ \mathcal{J}(\mathcal{A}^\mathcal{I})$.
\end{proof}

\begin{thm}\label{thm35}
Let $(X,\mathcal{F},\tau)$ be a PM space, $x=\{x_k\}_{k\in\mathbb{N}}$ be a sequence in $X$ and $\mathcal{I}$ be an ideal such that $\delta_{\mathcal{A}^\mathcal{I}}$ satisfies the property AP$\mathcal{A}^\mathcal{I}$O. Then $x_k\xrightarrow{\mathcal{A}^\mathcal{I}\mbox{-}st^{\mathcal{F}}}\mathcal{L}$ if and only if there exists a sequence $\{y_k\}_{k\in\mathbb{N}}$ such that $x_k=y_k$ for $a.a.k(\mathcal{A}^\mathcal{I})$ and $y_k\xrightarrow{\mathcal{F}}\mathcal{L}$.
\end{thm}
\begin{proof}
Let $x_k\xrightarrow{\mathcal{A}^\mathcal{I}\mbox{-}st^{\mathcal{F}}}\mathcal{L}$. Then we have $$\mathcal{I}\mbox{-}st_\mathcal{A}^{\mathcal{F}}\mbox{-}\lim\limits_{k\rightarrow\infty}d_L(\mathcal{F}_{x_k\mathcal{L}},\varepsilon_0)=0.$$ 
So by Theorem \ref{thm34}, there is a set $\mathcal{M}=\{j_1<j_2<...<j_n<...\}\subset\mathbb{N}$ such that $\delta_{\mathcal{A}^\mathcal{I}}(\mathcal{M})=1$ and $\mathcal{F}\mbox{-}\lim\limits_{n\rightarrow\infty}d_L(\mathcal{F}_{x_{j_n}\mathcal{L}},\varepsilon_0)=0$. 

We now construct a sequence $y=\{y_k\}_{k\in\mathbb{N}}$ as follows: 
\[ y_k = \left\{
  \begin{array}{l l}
    x_k, & \quad \text{if}~ k\in \mathcal{M} \\
    \mathcal{L}, & \quad \text{if}~ k\notin \mathcal{M}.
  \end{array} \right.\]\\
Then clearly, $y_k\xrightarrow{\mathcal{F}}\mathcal{L}$ and $x_k=y_k$ for $a.a.k(\mathcal{A}^\mathcal{I})$.

Conversely, let $x_k=y_k$ for $a.a.k(\mathcal{A}^\mathcal{I})$ and $y_k\xrightarrow{\mathcal{F}}\mathcal{L}$. Let $t>0$. Since $\mathcal{A}$ is a non-negative regular summability matrix so there exists an $N_0\in\mathbb{N}$ such that for each of $n\geq N_0$, we get

$$\sum\limits_{x_k\notin\mathcal{N}_\mathcal{L}(t)}a_{nk}\leq\sum\limits_{x_k\neq y_k}a_{nk}+\sum\limits_{y_k\notin\mathcal{N}_\mathcal{L}(t)}a_{nk}.$$
As $\{y_k\}_{k\in\mathbb{N}}$ is strongly convergent to $\mathcal{L}$, so the set $\{k\in\mathbb{N}:y_k\notin\mathcal{N}_\mathcal{L}(t)\}$ is finite and so $\delta_{\mathcal{A}^\mathcal{I}}(\{k\in \mathbb{N}:y_k\notin\mathcal{N}_\mathcal{L}(t)\})=0$.\\
Thus,
\begin{eqnarray*}
&& ~~~\delta_{\mathcal{A}^\mathcal{I}}(\{k\in \mathbb{N}:x_k\notin\mathcal{N}_\mathcal{L}(t)\})\\
&& \leq \delta_{\mathcal{A}^\mathcal{I}}(\{k\in \mathbb{N}:x_k\neq y_k\})+ \delta_{\mathcal{A}^\mathcal{I}}(\{k\in \mathbb{N}:y_k\notin\mathcal{N}_\mathcal{L}(t)\})=0.
\end{eqnarray*}
Therefore $\delta_{\mathcal{A}^\mathcal{I}}(\{k\in \mathbb{N}:x_k\notin\mathcal{N}_\mathcal{L}(t)\})=0$ i.e., the sequence $\{x_k\}_{k\in\mathbb{N}}$ is strongly $\mathcal{A}^\mathcal{I}$-statistically convergent to $\mathcal{L}$.
\end{proof}

\begin{defn}\label{def36}\cite{Ko1}
Let $(X,\rho)$ be a metric space and $x=\{x_k\}_{k\in\mathbb{N}}$ be a sequence in $X$. Then $x$ is said to be $\mathcal{I}$-Cauchy in $X$ if for every $\gamma>0$, there exists a natural number $k_0$ such that
$$\{k\in\mathbb{N}:\rho(x_k,x_{k_0})\geq\gamma\}\in\mathcal{I}.$$
\end{defn}
\begin{note}\label{no32}
If we take $\mathcal I = \mathcal{J}(\mathcal{A}^\mathcal{I})$, then the notion of $\mathcal{I}$-Cauchyness becomes the notion of $\mathcal{A}^\mathcal{I}$-Cauchyness.
\end{note}
Now we discuss the following lemma in a metric space which is needed to study some properties of strong $\mathcal{A}^\mathcal{I}$-statistical Cauchyness in PM spaces.

\begin{lem}\label{lem31}
Let $(X,\rho)$ be a metric space and $x=\{x_k\}_{k\in\mathbb{N}}$ be a sequence in $X$. Then the following statements are equivalent:
\begin{enumerate}
	\item $x$ is an $\mathcal{A}^\mathcal{I}$-statistically Cauchy sequence.
	\item  For all $\gamma>0$, there is a set $\mathcal{M}\subset\mathbb{N}$ such that $\delta_{\mathcal{A}^\mathcal{I}}(\mathcal{M})=0$ and $\rho(x_m,x_n)<\gamma$ for all $m,n\notin \mathcal{M}$.
	\item For every $\gamma>0$, $\delta_{\mathcal{A}^\mathcal{I}}(\{j\in\mathbb{N}:\delta_{\mathcal{A}^\mathcal{I}}(D_j)\neq 0\})=0$, where $D_j(\gamma)=\{k\in\mathbb{N}:\rho(x_k,x_j)\geq \gamma\}$, $j\in\mathbb{N}$.
\end{enumerate}
\end{lem}

\begin{thm}\label{thm36}
Let $(X,\mathcal{F},\tau)$ be a  PM space, $x=\{x_k\}_{k\in\mathbb{N}}$ be a sequence in $X$ and $\mathcal{I}$ be an admissible ideal in $\mathbb N$ such that $\delta_{\mathcal{A}^\mathcal{I}}$ has the property AP$\mathcal{A}^\mathcal{I}$O. Then $x$ is strongly $\mathcal{A}^\mathcal{I}$-statistically Cauchy sequence in $X$ if and only if $x$ is strongly $\mathcal{A}^{\mathcal{I}^*}$-statistically Cauchy sequence in $X$.
\end{thm}
\begin{proof}
Let $x=\{x_k\}_{k\in\mathbb{N}}$ be a sequence in $X$ such that $x$ is strongly $\mathcal{A}^\mathcal{I}$-statistically Cauchy sequence in $X$. Then for all $t>0$, there exists a natural number $k_0$ depending on $t$ such that the set $\{k\in\mathbb N:\mathcal{F}_{x_kx_{k_0}}(t)\leq 1-t\}$
has $\mathcal{A}^\mathcal{I}$-density zero. Then, $$\delta_{\mathcal{A}^\mathcal{I}}(\{k\in\mathbb{N}:d_L(\mathcal{F}_{x_kx_{k_0}},\varepsilon_0)\geq t\})=0.$$ 
Set
$G_1=\{k\in\mathbb{N}:d_L(\mathcal{F}_{x_kx_{k_0}},\varepsilon_0)\geq 1\}$, $G_j=\{k\in\mathbb{N}:\frac{1}{j-1}> d_L(\mathcal{F}_{x_kx_{k_0}},\varepsilon_0)\geq\frac{1}{j}\}$ for $ j \geq 2,~ j \in
\mathbb{N}$. Then $\{G_j\}_{j\in\mathbb N}$ is a sequence of mutually disjoint subsets of $\mathbb{N}$ with $\delta_{\mathcal{A}^\mathcal{I}}(G_j)=0$ for each $j\in\mathbb{N}$. Since $\delta_{\mathcal{A}^\mathcal{I}}$ satisfies the property $AP\mathcal{A}^\mathcal{I}O$ so there exists a sequence $\{H_j\}_{j\in\mathbb N}$ of subsets of $\mathbb{N}$ with $\left|G_j\Delta
H_j\right|<\infty$ and $\delta_{\mathcal{A}^\mathcal{I}}(H=\bigcup\limits_{j=1}^\infty H_j)=0$.
We claim that $\{x\}_ \mathcal{M}$ is a strongly Cauchy sequence in $X$ where $\mathcal{M}=\mathbb{N}\setminus H$. To prove our claim, let $\gamma>0$ be given. Choose $i\in \mathbb{N}$ so that $\frac{1}{i+1}<\gamma $. Then
$\{k\in\mathbb{N}: d_L(\mathcal{F}_{x_kx_{k_0}},\varepsilon_0)\geq \gamma\}\subset
\bigcup\limits_{j=1}^{i+1}G_j$. Since $\left|G_j\Delta
H_j\right|<\infty$ for all $j=1,2,...,i+1$, so there exists $n'\in\mathbb{N}$
such that
$\bigcup\limits_{j=1}^{i+1}G_j\cap(n',\infty)=\bigcup\limits_{j=1}^{i+1}H_j\cap(n',\infty)$.
Now if $k\notin H$, $k>n'$ then $k\notin
\bigcup\limits_{j=1}^{i+1}H_j$. And consequently, $k\notin
\bigcup\limits_{j=1}^{i+1}G_j$, which implies
$d_L(\mathcal{F}_{x_kx_{k_0}},\varepsilon_0)<\gamma$. Therefore, $x$ is strongly $\mathcal{A}^{\mathcal{I}^*}$-statistically Cauchy sequence in $X$.

Conversely, let $x$ be strongly $\mathcal{A}^{\mathcal{I}^*}$-statistically Cauchy sequence in $X$. Then there exists a subset $\mathcal{M}=\{q_1<q_2<...\}$ of $\mathbb{N}$ such
that $\delta_{\mathcal{A}^\mathcal{I}}(\mathcal{M})=1$ and $\{x\}_ \mathcal{M}$ is a strongly Cauchy sequence in $X$. Then for each $t>0$, there exists a natural number $k_0$ depending on $t$ such that 
$$ \mathcal{F}_{x_{q_n}x_{q_{k_0}}}(t)>1-t, \hspace{1in} \forall~ n\geq k_0,$$
i.e.,
$$d_L(\mathcal{F}_{x_{q_n}x_{q_{k_0}}},\varepsilon_0)<t, \hspace{1in} \forall~ n\geq k_0.$$
Let $E_t=\{n\in
\mathbb{N}:d_L(\mathcal{F}_{x_{q_n}x_{q_{k_0}}},\varepsilon_0)\geq
t\}$. Then $E_t\subset\mathbb{N} \setminus \{q_{ _{k_0+1}},q_{
_{k_0+2}},...\}$. Now $\delta_{\mathcal{A}^\mathcal{I}}(\mathbb{N} \setminus \{q_{
_{k_0+1}}, q_{ _{k_0+2}},...\}) = 0$ and so $\delta_{\mathcal{A}^\mathcal{I}}(E_t)=0$.

 Hence $x$ is strongly $\mathcal{A}^\mathcal{I}$-statistically Cauchy sequence in $X$.
\end{proof}

\begin{thm}\label{thm37}
Let $(X,\mathcal{F},\tau)$ be a PM space and $x=\{x_k\}_{k\in\mathbb{N}}$ be a sequence in $X$. If $x$ is strongly $\mathcal{A}^\mathcal{I}$-statistically convergent, then $x$ is strongly $\mathcal{A}^\mathcal{I}$-statistically Cauchy. 
\end{thm}

\begin{proof}
The proof directly follows from [Theorem 3.5 \cite{Se2}], by taking the ideal
$ \mathcal{J}(\mathcal{A}^\mathcal{I})$.
\end{proof}

\begin{cor}\label{cor31}
Let $(X,\mathcal{F},\tau)$ be a PM space, $x=\{x_k\}_{k\in\mathbb{N}}$ be a sequence in $X$ and $\mathcal{I}$ be an admissible ideal in $\mathbb N$ such that $\delta_{\mathcal{A}^\mathcal{I}}$ has the property AP$\mathcal{A}^\mathcal{I}$O. If $x$ is strongly $\mathcal{A}^\mathcal{I}$-statistically convergent, then $x$ has a strongly Cauchy subsequence in $X$. 
\end{cor}
\begin{proof}
Directly follows from Theorem \ref{thm36} and Theorem \ref{thm37}.
\end{proof}
\begin{thm}\label{thm38}
Let $(X,\mathcal{F},\tau)$ be a PM space and $x=\{x_k\}_{k\in\mathbb{N}}$ be a sequence in $X$. If the sequence $x$ is strongly $\mathcal{A}^\mathcal{I}$-statistically Cauchy, then for each $t>0$, there is a set $\mathcal{P}_t\subset \mathbb{N}$ with $\delta_{\mathcal{A}^\mathcal{I}}(\mathcal{P}_t)=0$ such that $\mathcal{F}_{x_mx_j}(t)>1-t$ for any $m,j\notin \mathcal{P}_t$.
\end{thm}

\begin{proof}
Let $x$ be strongly $\mathcal{A}^\mathcal{I}$-statistically Cauchy. Let $t>0$. Then, there is a $\gamma=\gamma(t)>0$ such that,
$$\mathcal{F}_{lr}(t)>1-t ~\text{whenever}~ \mathcal{F}_{lj}(\gamma)>1-\gamma ~\text{and}~ \mathcal{F}_{jr}(\gamma)>1-\gamma.$$

As $x$ is strongly $\mathcal{A}^\mathcal{I}$-statistically Cauchy, so there is an $k_0=k_0(\gamma)\in\mathbb{N}$ such that 
$$\delta_{\mathcal{A}^\mathcal{I}}(\{k\in \mathbb{N}:\mathcal{F}_{x_kx_{k_0}}(\gamma)\leq 1-\gamma\})=0.$$
Let $\mathcal{P}_t=\{m\in \mathbb{N}:\mathcal{F}_{x_mx_{k_0}}(\gamma)\leq 1-\gamma\}$. Then $\delta_{\mathcal{A}^\mathcal{I}}(\mathcal{P}_t)=0$ and $\mathcal{F}_{x_mx_{k_0}}(\gamma)>1-\gamma$ and $\mathcal{F}_{x_jx_{k_0}}(\gamma)>1-\gamma$ for $m,j\notin \mathcal{P}_t$.
Hence for every $t>0$, there is a set $\mathcal{P}_t\subset\mathbb{N}$ with $\delta_{\mathcal{A}^\mathcal{I}}(\mathcal{P}_t)=0$ such that $\mathcal{F}_{x_mx_j}(t)>1-t$ for every $m,j\notin \mathcal{P}_t$.
\end{proof}

\begin{cor}\label{cor72}
Let $(X,\mathcal{F},\tau)$ be a PM space, $x=\{x_k\}_{k\in\mathbb{N}}$ be a sequence in $X$. If $x$ is strongly $\mathcal{A}^\mathcal{I}$-statistically Cauchy, then for each $t>0$, there is a set $\mathcal{Q}_t\subset \mathbb{N}$ with $\delta_{\mathcal{A}^\mathcal{I}}(\mathcal{Q}_t)=1$ such that $\mathcal{F}_{x_mx_j}(t)>1-t$ for any $m,j\in \mathcal{Q}_t$.
\end{cor}

\begin{thm}\label{thm39}
Let $(X,\mathcal{F},\tau)$ be a PM space, $x=\{x_k\}_{k\in\mathbb{N}}$, $y=\{y_k\}_{k\in\mathbb{N}}$ be two strongly $\mathcal{A}^\mathcal{I}$-statistically Cauchy sequences in $X$. Then $\{\mathcal{F}_{{x_k}{y_k}}\}_{k\in\mathbb{N}}$ is an $\mathcal{A}^\mathcal{I}$-statistically Cauchy sequence in $(\mathcal{D}^+,d_L)$.
\end{thm}
\begin{proof}
As $x$ and $y$ are strongly $\mathcal{A}^\mathcal{I}$-statistically Cauchy sequences, so by corollary \ref{cor72}, for every $\gamma>0$ there are $\mathcal{U}_\gamma, \mathcal{V}_\gamma\subset\mathbb{N}$ with $\delta_{\mathcal{A}^\mathcal{I}}(\mathcal{U}_\gamma)=\delta_{\mathcal{A}^\mathcal{I}}(\mathcal{V}_\gamma)=1$ so that $\mathcal{F}_{x_mx_j}(\gamma)>1-\gamma$ holds for any $m,j\in \mathcal{U}_\gamma$ and $\mathcal{F}_{y_ny_z}(\gamma)>1-\gamma$ holds for any $n,z\in \mathcal{V}_\gamma$. Let $\mathcal{W}_\gamma=\mathcal{U}_\gamma\cap \mathcal{V}_\gamma$. Then $\delta_{\mathcal{A}^\mathcal{I}}(\mathcal{W}_\gamma)=1$. So, for every $\gamma>0$, there is a set $\mathcal{W}_\gamma\subset\mathbb{N}$ with $\delta_{\mathcal{A}^\mathcal{I}}(\mathcal{W}_\gamma)=1$ so that $\mathcal{F}_{x_px_q}(\gamma)>1-\gamma$ and $\mathcal{F}_{y_py_q}(\gamma)>1-\gamma$ for any $p,q\in \mathcal{W}_\gamma$. Now let $t>0$. Then there exists a $\gamma(t)$ and hence a set $\mathcal{W}_\gamma=\mathcal{W}_t\subset\mathbb{N}$ with $\delta_{\mathcal{A}^\mathcal{I}}(\mathcal{W}_t)=1$ so that $d_L(\mathcal{F}_{x_py_p},\mathcal{F}_{x_qy_q})<t$ for any $p,q\in \mathcal{W}_t$, as $\mathcal{F}$ is uniformly continuous. Then the result follows from Lemma \ref{lem31}.
\end{proof}

\section{\textbf{Strong $\mathcal{A}^\mathcal{I}$-statistical limit points and strong $\mathcal{A}^\mathcal{I}$-statistical cluster points}}\label{sec4}
In this section following the works of \c{S}en\c{c}imen et al. \cite{Se2} and Malik et al. \cite{Pr1, Pr2} we discuss some basic properties of strong $\mathcal{A}^\mathcal{I}$-statistical cluster points of a sequence in a PM space including their interrelationship.

\begin{defn}\cite{Sh5, Se}\label{def41}
Let $(X,\mathcal{F},\tau)$ be a PM space and $x=\{x_k\}_{k\in\mathbb{N}}$ be a sequence in $X$. An element $\mathcal{L}\in X$ is called a strong limit point of $x$, if there is a subsequence of $x$ that strongly converges to $\mathcal{L}$.
\end{defn}
To denote the set of all strong limit points of any sequence $x$ in a PM space $(X,\mathcal{F},\tau)$ we use the notation $L_x^\mathcal{F}$.

\begin{defn}\cite{Se2}\label{def42}
Let $(X,\mathcal{F},\tau)$ be a PM space and $\mathcal I$ be an
admissible ideal in $\mathbb{N}$ and $x=\{x_k\}_{k\in\mathbb{N}}$
be a sequence in $X$. An element $\zeta\in X$ is said to be a
strong $\mathcal{I}$-limit point of $x$, if there is a subset $Q=
\{q_1 < q_2 < ...\}$ of $\mathbb{N}$ such that $ Q \notin
\mathcal I$ and $\{x_{q_k}\}_{k \in \mathbb{N}}$ strongly
converges to $\zeta$.
\end{defn}

\begin{defn}\cite{Se2}\label{def43}
Let $(X,\mathcal{F},\tau)$ be a PM space, $\mathcal I$ be an
admissible ideal in $\mathbb{N}$ and $x=\{x_k\}_{k\in\mathbb{N}}$
be a sequence in $X$. An element $\eta\in X$ is said to be a strong
$\mathcal{I}$-cluster point of $x$, if for every $t>0$, the set
$\{k\in\mathbb N : x_k\in\mathcal{N}_\eta(t)\}) \notin \mathcal I$.
\end{defn}

\begin{note}\label{no41}
(i) If $\mathcal{I}=\mathcal{I}_{d}=\{\mathcal{K}\subset
\mathbb{N}: d(\mathcal{K})=0\}$, then in a PM space
the notions of strong $\mathcal {I}_{d}$-limit point and strong
$\mathcal {I}_{d}$-cluster point coincide with the notions of
strong statistical limit point \cite{Se} and strong statistical cluster point \cite{Se} respectively.

(ii) Let $\mathcal{I}$ be an admissible ideal in $\mathbb{N}$ then the notions of strong
$\mathcal{J}(\mathcal{A}^\mathcal{I})$-limit point and strong
$\mathcal{J}(\mathcal{A}^\mathcal{I})$-cluster point of sequences in a PM space
become the notions of strong $\mathcal{A}^\mathcal{I}$-statistical limit point and strong $\mathcal{A}^\mathcal{I}$-statistical cluster point respectively. Further, if $\mathcal{I}=\mathcal{I}_{fin}=\{\mathcal{K}\subset\mathbb{N}: \left|\mathcal{K}\right|<\infty\}$, then the notions of strong $\mathcal{J}(\mathcal{A}^\mathcal{I})$-limit point and strong $\mathcal{J}(\mathcal{A}^\mathcal{I})$-cluster point of sequences in a PM space coincide with strong $\mathcal{A}$-statistical limit point \cite{Pr3} and strong $\mathcal{A}$-statistical cluster point \cite{Pr3} respectively. 

(iii) If $\mathcal{A}$ be the Cesaro matrix $C_1$ and $\mathcal{I}$ is an admissible ideal,
then the notions of strong $\mathcal{J}({C_1}^\mathcal{I})$-limit point and
strong $\mathcal{J}({C_1}^\mathcal{I})$-cluster point of sequences in a PM space
become the notions of strong $\mathcal{I}$-statistical limit point and
strong $\mathcal{I}$-statistical cluster point respectively.
\end{note}

Let $(X,\mathcal{F},\tau)$ be a PM space, $x=\{x_k\}_{k\in\mathbb{N}}$ be a sequence in $X$. Let $\{x_{k_j}\}_{j\in\mathbb{N}}$ be a subsequence of $x$ and $\mathcal{K}=\{k_j\in\mathbb{N}:j\in\mathbb{N}\}$ then we denote $\{x_{k_j}\}_{j\in\mathbb{N}}$ by $\{x\}_\mathcal{K}$. Now, if $\delta_{\mathcal{A}^\mathcal{I}}(\mathcal{K}) = 0$, $\{x\}_\mathcal{K}$ is said to be an $\mathcal{A}^\mathcal{I}$-thin subsequence of $x$. On the other hand, $\{x\}_\mathcal{K}$ is said to be an $\mathcal{A}^\mathcal{I}$-nonthin subsequence of $x$, if $\mathcal{K}$ does not have $\mathcal{A}^\mathcal{I}$ density zero i.e., if either $\delta_{\mathcal{A}^\mathcal{I}}(\mathcal{K})$ is a positive number or, the $\mathcal{A}^\mathcal{I}$-density of $\mathcal{K}$ does not exist.

In view of Definition \ref{def42}, Definition \ref{def43} and Note \ref{no41}.(ii) we now restate the definitions of strong $\mathcal{A}^\mathcal{I}$-statistical limit point and strong $\mathcal{A}^\mathcal{I}$-statistical cluster point in a PM space.

\begin{defn}\cite{Se2}\label{def44}
Let $(X,\mathcal{F},\tau)$ be a PM space, $x=\{x_k\}_{k\in\mathbb{N}}$ be a sequence in $X$. An element $\zeta\in X$ is said to be a strong $\mathcal{A}^\mathcal{I}$-statistical limit point of $x$, if there is an $\mathcal{A}^\mathcal{I}$-nonthin subsequence of $x$ that strongly converges to $\zeta$. 
\end{defn}
To denote the set of all strong $\mathcal{A}^\mathcal{I}$-statistical limit points of any sequence $x=\{x_k\}_{k\in\mathbb{N}}$ in a PM space $(X,\mathcal{F},\tau)$ we use the notation $\Lambda_x^\mathcal{A}(\mathcal{I})^\mathcal{F}_{s}$.

\begin{defn}\cite{Se2}\label{def45}
Let $(X,\mathcal{F},\tau)$ be a PM space and $x=\{x_k\}_{k\in\mathbb{N}}$ be a sequence in $X$. An element $\nu\in X$ is said to be a strong $\mathcal{A}^\mathcal{I}$-statistical cluster point of $x$, if for every $t>0$, the set $\delta_{\mathcal{A}^\mathcal{I}}(\{k\in\mathbb N : \mathcal{F}_{x_k\nu}(t)>1-t\})$ does not equal to zero.
\end{defn}
To denote the set of all strong $\mathcal{A}^\mathcal{I}$-statistical cluster points of any sequence $x=\{x_k\}_{k\in\mathbb{N}}$ in a PM space $(X, \mathcal{F}, \tau)$ we use the notation $\Gamma_x^\mathcal{A}(\mathcal{I})^\mathcal{F}_{s}$.

\begin{thm}\label{thm41}
Let $(X,\mathcal{F},\tau)$ be a PM space and $x=\{x_k\}_{k\in\mathbb{N}}$ be a sequence in $X$. Then $\Lambda_x^\mathcal{A}(\mathcal{I})^\mathcal{F}_{s}\subset \Gamma_x^\mathcal{A}(\mathcal{I})^\mathcal{F}_{s}\subset L_x^\mathcal{F}$.
\end{thm}
\begin{proof}
The proof directly follows from [Theorem 4.1 \cite{Se2}], by taking the ideal
$ \mathcal{J}(\mathcal{A}^\mathcal{I})$.
\end{proof}

\begin{thm}\label{thm42}
Let $(X,\mathcal{F},\tau)$ be a PM space, $x=\{x_k\}_{k\in\mathbb{N}}$ be a sequence in $X$ and $\mathcal{I}$ be an ideal such that $\delta_{\mathcal{A}^\mathcal{I}}$ satisfies the property AP$\mathcal{A}^\mathcal{I}$O. If $\mathcal{I}\mbox{-}st_\mathcal{A}^{\mathcal{F}}\mbox{-}\lim\limits_{k\rightarrow \infty}x_k = \mu$, then $\Lambda_x^\mathcal{A}(\mathcal{I})^\mathcal{F}_{s}=\Gamma_x^\mathcal{A}(\mathcal{I})^\mathcal{F}_{s}=\{\mu\}$.
\end{thm}
\begin{proof}
Let $\mathcal{I}\mbox{-}st_\mathcal{A}^{\mathcal{F}}\mbox{-}\lim\limits_{k\rightarrow \infty}x_k = \mu$. So for every $t>0$, $\delta_{\mathcal{A}^\mathcal{I}}(\{k\in\mathbb{N}:\mathcal{F}_{x_k\mu}(t)>1-t\})=1$. Therefore, $\mu\in\Gamma_x^\mathcal{A}(\mathcal{I})^\mathcal{F}_{s}$. Now assume that there exists at least one $\alpha\in\Gamma_x^\mathcal{A}(\mathcal{I})^\mathcal{F}_{s}$ such that $\alpha\neq\mu$. Then $\mathcal{F}_{\alpha\mu}\neq\varepsilon_0$. Then there is a $t_1>0$ such that $d_L(\mathcal{F}_{\alpha\mu},\varepsilon_0)=t_1$. Let $t>0$ be given such that $d_L(\mathcal{F}_{pq},\varepsilon_0)<t$ and
$d_L(\mathcal{F}_{qr},\varepsilon_0)<t $ imply that
$d_L(\mathcal{F}_{pr},\varepsilon_0)<t_1$. Now since  $\mu,\alpha\in\Gamma_x^\mathcal{A}(\mathcal{I})^\mathcal{F}_{s}$, for that $t>0$, $\delta_{\mathcal{A}^\mathcal{I}}(\mathcal{K})\neq 0$ and $\delta_{\mathcal{A}^\mathcal{I}}(\mathcal{M})\neq 0$ where
$\mathcal{K}=\{k\in\mathbb{N}: \mathcal{F}_{x_k\mu}(t)> 1-t\}$ and $\mathcal{M}=\{k\in\mathbb{N}: \mathcal{F}_{x_k\alpha}(t)> 1-t\}$.
As, $\mu\neq\alpha$, so $\mathcal{K}\cap \mathcal{M}=\emptyset$ and so $\mathcal{M}\subset \mathcal{K}^c$.
Since $\mathcal{I}\mbox{-}st_\mathcal{A}^{\mathcal{F}}\mbox{-}\lim\limits_{k\rightarrow \infty}x_k= \mu$ so $\delta_{\mathcal{A}^\mathcal{I}}(\mathcal{K}^c)=0$. Hence $\delta_{\mathcal{A}^\mathcal{I}}(\mathcal{M})=0$, which is a contradiction. 

Therefore, $\Gamma_x^\mathcal{A}(\mathcal{I})^\mathcal{F}_{s}=\{\mu\}$. 

As $\mathcal{I}\mbox{-}st_\mathcal{A}^{\mathcal{F}}\mbox{-}\lim\limits_{k\rightarrow \infty}x_k= \mu$, so from Theorem \ref{thm35}, we have $\mu\in\Lambda_x^\mathcal{A}(\mathcal{I})^\mathcal{F}_{s}$. Now by Theorem \ref{thm41}, we get $\Lambda_x^\mathcal{A}(\mathcal{I})^\mathcal{F}_{s}=\Gamma_x^\mathcal{A}(\mathcal{I})^\mathcal{F}_{s}=\{\mu\}$.
\end{proof}

\begin{thm}\label{thm43}
Let $(X,\mathcal{F},\tau)$ be a PM space. Also let $x=\{x_k\}_{k\in\mathbb{N}}$ and $y=\{y_k\}_{k\in\mathbb{N}}$ be two sequences in $X$ such that $\delta_{\mathcal{A}^\mathcal{I}}(\{k\in \mathbb{N} : x_k \neq y_k\})=0$. Then $\Lambda_x^\mathcal{A}(\mathcal{I})^\mathcal{F}_{s}=\Lambda_y^\mathcal{A}(\mathcal{I})^\mathcal{F}_{s}$ and $\Gamma_x^\mathcal{A}(\mathcal{I})^\mathcal{F}_{s}= \Gamma_y^\mathcal{A}(\mathcal{I})^\mathcal{F}_{s}$.
\end{thm}
\begin{proof}
Let $\nu \in \Gamma_x^\mathcal{A}(\mathcal{I})^\mathcal{F}_{s}$ and $t>0$ be given. Let $\mathcal{C}=\{k\in\mathbb N: x_k= y_k\}$. Since $\delta_{\mathcal{A}^\mathcal{I}}(\mathcal{C})=1$, so $\delta_{\mathcal{A}^\mathcal{I}}(\{k\in\mathbb N:\mathcal{F}_{x_k\nu}(t)>1-t\}\cap \mathcal{C})$ is not zero. This gives $\delta_{\mathcal{A}^\mathcal{I}}(\{k\in\mathbb{N}:\mathcal{F}_{y_k\nu}(t)>1-t\})\neq 0$ and so $\nu \in \Gamma_y^\mathcal{A}(\mathcal{I})^\mathcal{F}_{s}$. Since $\nu \in \Gamma_x^\mathcal{A}(\mathcal{I})^\mathcal{F}_{s}$ is arbitrary, so $\Gamma_x^\mathcal{A}(\mathcal{I})^\mathcal{F}_{s}\subset \Gamma_y^\mathcal{A}(\mathcal{I})^\mathcal{F}_{s}$. By similar argument, we get $ \Gamma _x^\mathcal{A}(\mathcal{I})^\mathcal{F}_{s} \supset \Gamma_y^\mathcal{A}(\mathcal{I})^\mathcal{F}_{s}$. Hence, $\Gamma_x^\mathcal{A}(\mathcal{I})^\mathcal{F}_{s}= \Gamma_y^\mathcal{A}(\mathcal{I})^\mathcal{F}_{s}$.

Now let $\mu\in \Lambda_y^\mathcal{A}(\mathcal{I})^\mathcal{F}_{s}$. Then $y$ has an $\mathcal{A}^\mathcal{I}$-nonthin subsequence $\{y_{k_j}\}_{j\in \mathbb N}$ that strongly converges to $\mu$. Let $\mathcal{M}=\{k_j \in \mathbb{N} : y_{k_j} =  x_{k_j}\}$. Since $\delta_{\mathcal{A}^\mathcal{I}}(\{k_j\in \mathbb{N} : y_{k_j}\neq x_{k_j}\})=0$ and $\{y_{k_j}\}_{j\in \mathbb{N}}$ is an $\mathcal{A}^\mathcal{I}$-nonthin subsequence of $y$ so $\delta_{\mathcal{A}^\mathcal{I}}(\mathcal{M}) \neq 0$. Now using the set $\mathcal{M}$ we get an $\mathcal{A}^\mathcal{I}$-nonthin subsequence $\{x\}_{\mathcal{M}'}$ of $x$ that strongly converges to $\mu$. Thus $\mu \in \Lambda_x^\mathcal{A}(\mathcal{I})^\mathcal{F}_{s}$. As $\mu \in \Lambda_y^\mathcal{A}(\mathcal{I})^\mathcal{F}_{s}$ is arbitrary, so $\Lambda_y^\mathcal{A}(\mathcal{I})^\mathcal{F}_{s}\subset \Lambda_x^\mathcal{A}(\mathcal{I})^\mathcal{F}_{s}$. Similarly, we have $\Lambda_x^\mathcal{A}(\mathcal{I})^\mathcal{F}_{s}\subset \Lambda _y^\mathcal{A}(\mathcal{I})^\mathcal{F}_{s}$. Therefore $\Lambda_x^\mathcal{A}(\mathcal{I})^\mathcal{F}_{s}= \Lambda_y^\mathcal{A}(\mathcal{I})^\mathcal{F}_{s}$.
\end{proof}

\begin{thm}\label{thm44}
Let $(X,\mathcal{F},\tau)$ be a PM space and $x=\{x_k\}_{k\in\mathbb{N}}$ be a sequence in $X$. Then the set $\Gamma_x^\mathcal{A}(\mathcal{I})^\mathcal{F}_{s}$ is a strongly closed set.
\end{thm}
\begin{proof}
The proof directly follows from [Theorem 4.2 \cite{Se2}], by taking the ideal
$ \mathcal{J}(\mathcal{A}^\mathcal{I})$.
\end{proof}

\begin{thm}\label{thm45}
Let $(X,\mathcal{F},\tau)$ be a PM space, $x=\{x_k\}_{k\in\mathbb{N}}$ be a sequence in $X$ and $\mathcal{C}$ be a strongly compact subset of $X$ such that $\mathcal{C}\cap\Gamma_x^\mathcal{A}(\mathcal{I})^\mathcal{F}_{s}=\emptyset$. Then $\delta_{\mathcal{A}^\mathcal{I}}(\mathcal{M})=0$, where $\mathcal{M}=\{k\in\mathbb{N}:x_k\in \mathcal{C}\}$.
\end{thm}
\begin{proof}
As $\mathcal{C}\cap\Gamma_x^\mathcal{A}(\mathcal{I})^\mathcal{F}_{s}=\emptyset$, so for all $\beta\in \mathcal{C}$, there exists a real number $t=t(\beta)>0$ so that $\delta_{\mathcal{A}^\mathcal{I}}(\{k\in\mathbb{N}:\mathcal{F}_{x_k\beta}(t)>1-t\})=0$. Let $\mathcal{N}_{\beta}(t)=\{q\in X:\mathcal{F}_{q\beta}(t)>1-t\}$. Then the family of strongly open sets $\mathcal{Q}=\{\mathcal{N}_{\beta}(t):\beta\in \mathcal{C}\}$ forms a strong open cover of $\mathcal{C}$. As $\mathcal{C}$ is a strongly compact set, so there exists a finite subcover $\{\mathcal{N}_{\beta_1}(t_1),\mathcal{N}_{\beta_2}(t_2),...,\mathcal{N}_{\beta_m}(t_m)\}$ of the strong open cover $\mathcal{Q}$. Then $\mathcal{C}\subset\bigcup\limits_{j=1}^m\mathcal{N}_{\beta_j}(t_j)$ and also for each $j=1,2,...,m$ we have $\delta_{\mathcal{A}^\mathcal{I}}(\{k\in\mathbb{N}:\mathcal{F}_{x_k\beta_j}(t_j)>1-t_j\})=0$. So we get for every $n\in\mathbb{N}$, 
$$\sum\limits_{x_k\in \mathcal{C}}a_{nk}\leq\sum\limits_{j=1}^m\sum\limits_{x_k\in \mathcal{N}_{\beta_j}(t_j)}a_{nk}.$$ Then by the property of $\mathcal{I}$ convergence,\\
\begin{eqnarray*}
&&~~~\mathcal{I}\mbox{-}\lim\limits_{n\rightarrow\infty}\sum\limits_{x_k\in \mathcal{C}}a_{nk}\leq\sum\limits_{j=1}^m\mathcal{I}\mbox{-}\lim\limits_{n\rightarrow\infty}\sum\limits_{x_k\in \mathcal{N}_{\beta_j}(t_j)}a_{nk}=0.\\
\end{eqnarray*}
This gives $\delta_{\mathcal{A}^\mathcal{I}}(\{k\in\mathbb{N}:x_k\in \mathcal{C}\})=0$.
\end{proof}

\begin{thm}\label{thm46}
Let $(X,\mathcal{F},\tau)$ be a PM space and $x=\{x_k\}_{k\in\mathbb{N}}$ be a sequence in $X$. If $x$ has a strongly bounded $\mathcal{A}^\mathcal{I}$-nonthin subsequence then the set $\Gamma_x^\mathcal{A}(\mathcal{I})^\mathcal{F}_{s}$ is non-empty and strongly closed.
\end{thm}
\begin{proof}
Let $\{x\}_\mathcal{M}$ be a strongly bounded $\mathcal{A}^\mathcal{I}$-nonthin subsequence of $x$. So $\delta_{\mathcal{A}^\mathcal{I}}(\mathcal{M})\neq 0$ and there exists a strongly compact subset $\mathcal{C}$ of $X$ such that $x_k\in \mathcal{C}$ for all $k\in \mathcal{M}$. If $\Gamma_x^\mathcal{A}(\mathcal{I})^\mathcal{F}_{s}=\emptyset$ then $\mathcal{C}\cap\Gamma_x^\mathcal{A}(\mathcal{I})^\mathcal{F}_{s}=\emptyset$ and then by Theorem \ref{thm45}, we get $\delta_{\mathcal{A}^\mathcal{I}}(\{k\in\mathbb{N}:x_k\in \mathcal{C}\})=0$. Since $\mathcal{A}$ is a non-negative regular summability matrix so there exists an $N_0\in\mathbb{N}$ such that for every $n\geq N_0$ we have
$$\sum\limits_{k\in\mathcal{M}}a_{nk}\leq \sum\limits_{x_k\in \mathcal{C}}a_{nk}$$ and this gives $\delta_{\mathcal{A}^\mathcal{I}}(\mathcal{M})=0$, which contradicts our assumption. Hence $\Gamma_x^\mathcal{A}(\mathcal{I})^\mathcal{F}_{s}$ is nonempty and also by Theorem \ref{thm44}, $\Gamma_x^\mathcal{A}(\mathcal{I})^\mathcal{F}_{s}$ is strongly closed.
\end{proof}

\begin{defn}\label{def46}
Let $(X,\mathcal{F},\tau)$ be a PM space, $x=\{x_k\}_{k\in\mathbb{N}}$ be a sequence in $X$. Then $x$ is said to be strongly $\mathcal{A}^\mathcal{I}$-statistically bounded if there exists a strongly compact subset $\mathcal{C}$ of $X$ such that $\delta_{\mathcal{A}^\mathcal{I}}(\{k\in\mathbb{N}:x_k\notin \mathcal{C}\})=0$. 
\end{defn}

\begin{thm}\label{thm47}
Let $(X,\mathcal{F},\tau)$ be a PM space, $x=\{x_k\}_{k\in\mathbb{N}}$ be a sequence in $X$. If $x$ is strongly $\mathcal{A}^\mathcal{I}$-statistically bounded then the set $\Gamma_x^\mathcal{A}(\mathcal{I})^\mathcal{F}_{s}$ is nonempty and strongly compact.
\end{thm}
\begin{proof}
Let $\mathcal{C}$ be a strongly compact set with $\delta_{\mathcal{A}^\mathcal{I}}(\mathcal{V})=0$, where $\mathcal{V}=\{k\in\mathbb{N}:x_k\notin \mathcal{C}\}$. Then $\delta_{\mathcal{A}^\mathcal{I}}(\mathcal{V}^c)=1\neq 0$ and so $\mathcal{C}$ contains a bounded $\mathcal{A}^\mathcal{I}$- nonthin subsequence of $x$. So, by Theorem \ref{thm46}, $\Gamma_x^\mathcal{A}(\mathcal{I})^\mathcal{F}_{s}$ is nonempty and strongly closed. We now prove that $\Gamma_x^\mathcal{A}(\mathcal{I})^\mathcal{F}_{s}$ is strongly compact. For this we only show that $\Gamma_x^\mathcal{A}(\mathcal{I})^\mathcal{F}_{s}\subset C$. If possible let $\alpha\in \Gamma_x^\mathcal{A}(\mathcal{I})^\mathcal{F}_{s}\setminus \mathcal{C}$. As $\mathcal{C}$ is strongly compact so there is a $q>0$ such that $\mathcal{N}_\alpha(q)\cap \mathcal{C}=\emptyset$. So we get $\{k\in\mathbb{N}:\mathcal{F}_{x_k\alpha}(q)>1-q\}\subset\{k\in\mathbb{N}:x_k\notin \mathcal{C}\}$ which implies that $\delta_{\mathcal{A}^\mathcal{I}}(\{k\in\mathbb{N}:\mathcal{F}_{x_k\alpha}(q)>1-q\})=0$, a contradiction to our assumption that $\alpha\in\Gamma_x^\mathcal{A}(\mathcal{I})^\mathcal{F}_{s}$. So, $\Gamma_x^\mathcal{A}(\mathcal{I})^\mathcal{F}_{s}\subset \mathcal{C}$.

Therefore the set $\Gamma_x^\mathcal{A}(\mathcal{I})^\mathcal{F}_{s}$ is nonempty and strongly compact.
\end{proof}

\noindent\textbf{Acknowledgment:} The second author is grateful to
Council of Scientific and Industrial Research, India for his
fellowship funding under CSIR-JRF (SRF) scheme during the preparation of this paper.
\\


\begin{thebibliography}{99}

\bibitem{Ag1} A. Ghosh and S. Das, Strongly $\lambda$-statistically and strongly Vall\'{e}e-Poussin pre-Cauchy sequences in probabilistic metric spaces, \textit{Tamkang J. Math.}, 53 (2021). DOI:10.5556/j.tkjm.53.2022.3893

\bibitem{Ag2} A. Ghosh and S. Das, Some further studies on strong $\mathcal{I}_\lambda$-statistical convergence in probabilistic metric spaces, \textit{Analysis}, 42(1) (2022), 11–22.

\bibitem{Br} A. Bartoszewicz, P. Das and S. Glab, On matrix summability of spliced sequences and $A$-density of
points, \textit{Linear Algebra Appl.}, 487 (2015), 22-42.

\bibitem{Co4} J. Connor and J. Kline, On statistical limit points and the consistency of statistical convergence, \textit{J. Math. Anal. Appl.}, 197 (1996), 392-399.


\bibitem{Da5} P. Das, K. Dutta, V. Karakaya and S. Ghosal, On some further generalizations of strong convergence in probabilistic metric spaces using ideals, \textit{Abstract and App. Anal.}, DOI: 10.1155/2013/765060, (2013).

\bibitem{De1} K. Demirci, A-statistical core of a sequence, \textit{Demonstratio Mathematica}, 33(2)(2000), 343-354.

\bibitem{De3} K. Demirci, On A-statistical cluster points, \textit{Glas. Mat.}, 37(57) (2002), 293 – 301.

\bibitem{Dem} K. Dems, On $\mathcal{I}$-Cauchy sequences, \textit{Real Analysis Exchange}, 30(1): 123-128.

\bibitem{Du1} K. Dutta, P. Malik and M. Maity, Statistical Convergence of Double Sequences in Probabilistic Metric Spaces, \textit{Sel\c{c}uk J. Appl. Math.},  14(1) (2013), 57-70.

\bibitem{Ed1} O.H.H. Edely  and M. Mursaleen, On statistical $\mathfrak{A}$-Cauchy and statistical $\mathfrak{A}$-summability via ideal, \textit{Journal of Inequalities and Applications}, (2021) 
https://doi.org/10.1186/s13660-021-02564-4.

\bibitem{Ed2} O.H.H. Edely, On some properties of ${A^{I}}$-summability and $A^{{I^*}}$-summability, {Azerbaijan J. Math.}, 11(1) (2021), 189-200. 

\bibitem{Fa} H. Fast, Sur la convergence statistique, \textit{Colloq. Math}, 2 (1951), 241-244.

\bibitem{Fd} A. R. Freedman and I. J. Sember, Densities and
summability, \textit{Pacific J. Math.}, 95 (1981), 293-305.

\bibitem{Fr1} J. A. Fridy, On statistical convergence, \textit{Analysis},  5 (1985), 301-313.

\bibitem{Fr2} J. A. Fridy, Statistical limit points, \textit{Proc. Amer. Math. Soc.}, 118(4) (1993), 1187-1192.

\bibitem{Gu1} M. G\"{u}rdal and H. Sari, Extremal $A$-statistical limit points via ideals, \textit{Journal of the Egyptian Mathematical Society}, 22 (2014), 55-58.

\bibitem{Kl1} E. Kolk, The statistical convergence in Banach
spaces, \textit{Acta Comment. Univ. Tartu}, 928 (1991), 41-52.

\bibitem{Ko1} P. Kostyrko, T. \v{S}al\'{a}t and W. Wilczy\'{n}ski, {{\textit{I}}-convergence}, \textit{Real Anal.Exchange}, 26(2) (2000/2001), 669-685.

\bibitem{Ko2} P. Kostyrko, M. macaz, T. \v{S}al\'{a}t and M. Sleziak, {\textit{I}-convergence and
external \textit{I}-limit points}, \textit{Math. Slovaca}, 55(4) (2005), 443-454.

\bibitem{La1} B. K. Lahiri and P. Das, $I$ and $I^*$-convergence in topological spaces, \textit{Math. Bohemica}, 126 ((2005)), 153-160.

\bibitem{La2} B. K. Lahiri and P. Das, $I$ and $I^*$-convergence of nets, \textit{Real Analysis Exchange}, 33(2) (2007/2008), 431-442.

\bibitem{Pr1} P. Malik and S. Das, Further results on strong $\lambda$-statistical convergence of sequences in probabilistic metric spaces, \textit{Bol. Soc. Paran. Mat.}, to appear.

\bibitem{Pr2} P. Malik and S. Das, $\mathcal{I}_\lambda$-statistical limit point and $\mathcal{I}_\lambda$-statistical cluster point, \textit{Malaya Journal of Matematik}, 9(1) (2021), 331-337.

\bibitem{Pr3} P. Malik and S. Das, On strong $\mathcal{A}$-statistical convergence in probabilistic metric spaces, \textit{arXiv}, (2022). https://doi.org/10.48550/arXiv.2204.02727.

\bibitem{Pr4} P. Malik and S. Das, $\mathcal{A}^\mathcal{I}$-statistical limit points and $\mathcal{A}^\mathcal{I}$-statistical cluster points, \textit{Filomat}, 36(5) (2022).

\bibitem{Me} K. Menger, Statistical metrics, \textit{ Proc. Nat. Acad. Sci. USA}, 28 (1942), 535-537.

\bibitem{Pe} S. Pehlivan, A. G$\ddot{\text{u}}$ncan and M. A. Mamedov, \textit{Statistical cluster points of sequences in finite dimensional spaces}, Czechoslovak Mathematical Journal, 54(129) (2004), 95-102.

\bibitem{Sa} T. \v{S}al\'{a}t, On statistically convergent sequences of real numbers. \textit{Math. Slovaca} 30 (1980), 139-150.

\bibitem{Sav3} E. Savas, P. Das and S. Dutta, A note on strong matrix summability via ideals, \textit{Appl. Math. Lett.}, 25 (2012), 733-738.

\bibitem{Sav6} E. Savas, P. Das and S. Dutta, A note on some generalized summability methods, \textit{Acta Mathematica Universitatis Comenianae}, 82(2)(2017) 297-304.

\bibitem{Sc} I. J. Schoenberg, The integrability of certain
functions and related summability methods, \textit{ Amer. Math. Monthly}, 66 (1959), 361-375.

\bibitem{Sh1} B. Schweizer and A. Sklar, Statistical metric spaces, \textit{Pacific J. Math.}, 10 (1960), 314-334.

\bibitem{Sh2} B. Schweizer, A. Sklar, and E. Thorp, The metrization of statistical metric spaces, \textit{Pacific J. Math.}, 10 (1960), 673-675.

\bibitem{Sh3} B. Schweizer and A. Sklar, Statistical metric spaces arising from sets of random variables in Euclidean n-space, \textit{Theory of probability and its Applications}, 7 (1962), 447-456.

\bibitem{Sh4} B. Schweizer and A. Sklar, Tringle inequalities in a class of statistical metric spaces, \textit{J. London Math. Soc.}, 38 (1963), 401-406.

\bibitem{Sh5} B. Schweizer and A. Sklar, Probabilistic Metric Spaces, \textit{ North Holland: New York, Amsterdam, Oxford}, 1983.

\bibitem{Se} C. \c{S}en\c{c}imen and S. Pehlivan, Strong statistical convergence in probabilistic metric spaces, \textit{Stoch. Anal. Appl.},  26 (2008), 651-664.

\bibitem{Se2} C. \c{S}en\c{c}imen and S. Pehlivan, Strong ideal convergence in probabilistic metric
spaces, \textit{Proc. Indian Acad. Sci. (Math. Sci.)},  119(3) (2009), 401-410.

\bibitem{St} H. Steinhus, Sur la convergence ordinatre et la convergence asymptotique, \textit{Colloq. Math.}, 2 (1951), 73-74.

\bibitem{Tar}  R. M. Tardiff, Topologies for Probabilistic Metric spaces, \textit{Pacific J. Math.}, 65 (1976), 233-251.

\bibitem {Th} E. Thorp, Generalized topologies for statistical
metric spaces, \textit{Fundamenta Mathematicae}, 51 (1962), 9-21.

\end{thebibliography}
\end{document}